\documentclass[journal]{IEEEtran}
\usepackage{graphics} 
\usepackage{epsfig} 
\usepackage{times} 
\usepackage{amsmath} 
\usepackage{amssymb}  
\usepackage{amsmath} 
\usepackage{amssymb}  
\usepackage{color} 
\usepackage{hyperref}
\usepackage{subfigure,multirow}

\newtheorem{theorem}{Theorem}
\newtheorem{lem}{Lemma}
\newtheorem{defn}{Definition}

\newtheorem{prop}{Proposition}

\newenvironment{proof}{\quad{\it Proof:\ }}{ \hfill \IEEEQED\par}
\newenvironment{proof-of}[1]{\quad{\it Proof of #1:\ }}{ \hfill \IEEEQED\par}
\IEEEoverridecommandlockouts

\renewcommand{\Pr}{{\mathbb{P}}}
\newcommand\E{\ensuremath{\mathbb E}}
\newcommand\R{\ensuremath{\mathbb R}}

\newcommand{\cG}{{\mathcal{G}}}
\newcommand{\Hinf}{{\mathbb{H}_{\infty}}}
\newcommand{\Htwo}{{\mathbb{H}_{2}}}
\newcommand{\cC}{{\mathcal{C}}}
\newcommand{\cS}{{\mathcal{S}}}
\newcommand{\cD}{{\mathcal{D}}}
\newcommand{\cK}{{\mathcal{K}}}
\newcommand{\cDDp}{{\mathcal{DD}_+}}
\newcommand{\cDD}{{\mathcal{DD}}}
\newcommand{\cN}{{\mathcal{N}}}
\newcommand{\cO}{{\mathcal{O}}}
\newcommand{\diag}[1]{\textrm{diag}\left\{#1 \right\}}

\newcommand{\cM}{{\mathcal M}}
\newcommand{\tr}{\text{trace}}
\newcommand{\mH}{\mathcal H}
\newcommand{\mHp}{\mathcal H_{+}}

\newcommand{\cov}{{\text{cov}}}

\title{Structured Projection-Based Model Reduction with Application to Stochastic Biochemical Networks} 

\author{Aivar Sootla and James Anderson
\thanks{A. Sootla is with the Montefiore Institute, University of Li\`ege, B28, Li\`ege Belgium, B4000 e-mail: \texttt{asootla@ulg.ac.be}}
\thanks{J. Anderson is with St John's College, Oxford and the Department of Engineering Science, University of Oxford, Parks Road, Oxford, OX1 3PJ, U.K. e-mail: \texttt{james.anderson@eng.ox.ac.uk}}
\thanks{The authors would like thank Professors Bayu Jayawardhana and Shodhan Rao for kindly providing the kinetic model of yeast glycolysis. Dr Anderson acknowledges funding through a junior research fellowship from St. John's College, Oxford. Part of this work was performed when Dr Sootla was a postdoctoral researcher at Imperial College London. Therefore, Dr Sootla gratefully acknowledges support by the EPSRC Grants EP/J014214/1, EP/G036004/1 and is currently funded by the F.R.S-FNRS fellowship.}
}

\begin{document}

\maketitle
\IEEEpeerreviewmaketitle
\begin{abstract}
The Chemical Master Equation (CME) is well known to provide the highest resolution models of a biochemical reaction network. Unfortunately, even simulating the CME can be a challenging task. For this reason more simple approximations to the CME have been proposed. In this work we focus on one such model, the Linear Noise Approximation. Specifically, we consider implications of a recently proposed LNA time-scale separation method. We show that the reduced order LNA converges to the full order model in the mean square sense. Using this as motivation we derive a network structure preserving reduction algorithm based on structured projections. We present convex optimisation algorithms that describe how such projections can be computed and we discuss when structured solutions exits. We also show that for a certain class of systems, structured projections can be found using basic linear algebra and no optimisation is necessary. The algorithms are then applied to a linearised stochastic LNA model of the yeast glycolysis pathway.
\end{abstract}


\section{Introduction - Models}
Stochasticity is inherent in biochemical networks. The most general (and unfortunately, most complex) model that best encapsulates the behaviour of such a network is the \emph{Chemical Master Equation} (CME). The CME is a continuous time, infinite dimensional Markov Chain that describes the evolution of a probability density function of the concentrations of the species of all reactants in a given biochemical reaction.

The CME models a reaction network comprising of $R$ reactions and $N$ species evolving in a compartment of fixed volume $\Omega$ and takes the form:
\begin{equation}\label{eq:CME}
\frac{\partial \Pr(n,t)}{\partial t }=\Omega \sum_{i=1}^R \hat f(n-S_i,\Omega)-\hat f(n,\Omega))\Pr(n,t)
\end{equation}
where the vector $n = [n_i ,\hdots, n_N]^\ast$ indicates the total number of molecules of each species in the volume $\Omega$, $[\cdot]^\ast$ denotes transposition, $\hat f$ is the flux-vector and $S \in \R^{N \times R}$ is the stoichiometry matrix (the $i^{\text{th}}$ column of $S$ is denoted by $S_i$). Finally, $\Pr(n,t)$ is the probability that at time $t$ the number of molecules of each species is given by $n$, i.e. $\Pr(n,t)$ is the probability density at time $t$.

It is fairly clear from the form of \eqref{eq:CME} that the complexity of even simulating the CME quickly becomes intractable for all but the most simple of networks. A great deal of research has focussed on efficient methods for simulating the CME with the most popular being Gillespie's famous \emph{stochastic simulation algorithm}~\cite{gillespie1977exact} and the computationally more efficient version known as the \emph{$\tau$-leaping} algorithm~\cite{Gil01}. The reader is directed to \cite{Gil01} for a detailed description of both algorithms and their derivations. Related work that aims to approximate the CME that does not rely on time-scale arguments is described in \cite{MunK06}. In one of the extensions of $\tau$-leaping~\cite{cao2008slow}, it is proposed to replace propensities, which correspond to fast, in some sense, reactions, by their averages and simulate only ``slow'' reactions. The idea of averaging (or integrating) out a part of a stochastic process as a model reduction tool dates back to Khasminskii~\cite{khas1968aver_c} and can be traced to Krylov-Bogolyubov averaging methods~\cite{KrylovBogolyubov1937}.

Recent developments in systems and synthetic biology revived the interest in time-scale separation of biochemical networks. In~\cite{thomas2012rigorous}, the authors derived a time-scale separation method for the so-called Linear Noise Approximation (LNA) of a CME. The LNA is a Gaussian process, which approximates a CME, under the assumption of a large number of reactions occurring in a large volume. We will revisit this method in what follows, and argue that this is an averaging of a linear, time varying, stochastic differential equation (SDE). In~\cite{herathmodel}, the authors applied the classical Tikhonov theorem (cf.~\cite{kokotovic1999singular})  
to the moments of the Chemical Langevin Equation (a nonlinear SDE approximation of the CME). It can be argued that the averaging is implicitly applied, while computing these moments. Finally, stochastic averaging for CMEs has been recently proposed in~\cite{kang2013separation}. These results are based on averaging schemes for semimartingales~\cite{kurtz1992averaging}.

In the control literature, time-scale separation methods were phased out by the use of the so-called projection-based methods (cf.~\cite{AntoulasBook}). Projection methods were specifically derived for input-output systems and typically provide better (in terms of the $L_2$ gain) approximations. The main idea of these methods is to project the state space onto a lower dimensional space in such a way that the input-output behaviour of the approximated system is similar to the original one. In the context of linear stochastic differential equations (SDEs) we can use the intrinsic noise (the Brownian motion driving the process) as an input~\cite{hartmann2011balanced}. This idea also appears in the so-called low-noise limit results in stochastic calculus (cf.~\cite{freidlin2012random}). 

A caveat in using balancing, is that projections typically destroy any physical interpretation of the state space, which is not desirable in many applications, especially when analysing networked systems. One cure for this problem relies on graph partitioning and clustering algorithms~\cite{AndP12,MonTC14}, which, unfortunately, do not provide error bounds if the nodes have dynamics beyond simple integrators. An alternative approach was proposed in~\cite{Sandberg09}, where only a part of the state-space can be projected to a lower dimensional space. Even though it is still an open question as to when structured projections exist and can be computed, in some cases existence and polynomial-time computation can be guaranteed. For instance, in the case of positive directed networks, model reduction can be performed with trivial projections~\cite{Sootla2012positive}. A more sophisticated projection approach was recently proposed in~\cite{ishizaki2015clustered}. 

\emph{Contributions.} In previous work~\cite{sootla2014projectionI,sootla2014projectionII}, we derived the main idea of the algorithm and showed that it can be applied to the so-called monotone systems. In this paper, we extend these results in several directions. We show 
that the averaged LNA converges to the full order LNA in the mean square sense. This extends the original method~\cite{thomas2012rigorous}, where the convergence in distribution is argued. We provide a novel proof for error bounds for structured balanced truncation~\cite{Sandberg09}. From our point of view our proof is simpler and offers a different insight into the problem. Using this new technique we derive error bounds for structured balanced singular perturbation, which are used to justify our approach.
We continue by identifying a broad class of systems, to which structured model reduction can be applied, by using ideas from matrix theory. The property allowing this application is related to diagonal dominance of the drift matrix (see~\cite{willems1976lyapunov} for control theoretic implications). 
Finally, we present a projection-based model reduction algorithm for the LNA, which can be applied to a broad class of biochemical networks.

\textbf{\emph{Organisation.}} The paper is organised as follows. In Section~\ref{s:prel}, we briefly introduce averaging techniques for nonlinear deterministic and stochastic dynamical systems. In Section~\ref{s:aver-lna}, we present the averaging result for the LNA and provide the convergence proof. We then discuss structured model reduction of linear systems in Section~\ref{s:mor}, while providing sufficient conditions for computing the reduced order models and computing the error bounds for the structured model reduction. In Section~\ref{s:proj-lna} we apply structured model reduction techniques to the LNA and illustrate the application on examples in Section~\ref{s:ex}. Technical lemmas are found in Appendix. 

\textbf{\emph{Notation}}. $A^\ast$ denotes the complex conjugate transpose of the matrix $A$. The norm $\|\cdot\|_2$ is the standard induced matrix norm. $A \ge 0$ ($A\gg 0$) denotes that $a_{i j}\ge 0$  (resp., $a_{i j}> 0$) for all $i$, $j$. When $A$ is square symmetric, $A \succeq 0$ denotes that $A$ is positive semidefinite. $\E(\xi)$ and $\cov(\xi)$ stand for the mean and the covariance of the random variable $\xi$, respectively, while $\xi \in\cN(\mu, \Sigma)$ indicates that $\xi$ is drawn from a Gaussian distribution with mean $\mu$ and  covariance matrix $\Sigma$. 

\section{Preliminaries \label{s:prel}}
\subsection{Averaging \label{ss:ave}}
We first present some background material on averaging beginning with the deterministic case. Given a continuous, bounded function $g(t,x)$ where $g:[0,\infty)\times \cD \rightarrow \R^n $, $g$ is said to have an (ergodic) average denoted by $\widetilde{g}$ if
\begin{equation*}
\widetilde{g}(x) := \lim_{T\rightarrow \infty} \frac{1}{T}\int_{t}^{t+T}g(\tau,x)d\tau < \infty
\end{equation*}
and 
\begin{align*}
\left\| \frac{1}{T}\int_{t}^{t+T}g(\tau,x)d\tau -\widetilde{g}(x) \right\| & \le k \chi(T)
\end{align*}
for all $(t,x)$ in $[0,\infty) \times \cD_*$ and every compact set $\cD_* \subset \cD$, where $k$ is a positive constant and $\chi$ is a strictly decreasing, continuous, bounded function such that 
$\lim_{T\rightarrow \infty}\chi(T)=0$. We will refer to $\chi$ as the convergence function of $g$.
The most basic averaging problem formulation is as follows. Consider the deterministic autonomous system
\begin{equation}\label{eq:ave}
\dot{x} = \epsilon f(t,x,\epsilon)
\end{equation}
defined on a domain $\cD$ and with $0< \epsilon \ll 1$. Assume that the first and second partial derivatives of $f$ with respect to $x$ and $\epsilon$ are continuous and bounded on $[0,\infty)\times \cD_* \times [0,\epsilon]$ for every compact set $\cD_* \subset \cD$. Suppose that $f(t,x,0)$ has an average function $\widetilde{f}$ on $[0,\infty)\times \cD$, then 
\begin{equation}\label{eq:ave_sys}
\dot{x} = \epsilon \widetilde{f}(x)
\end{equation}
is said to be the average system induced by \eqref{eq:ave}. Let $x(t,\epsilon)$ and $\widetilde{x}(t\epsilon)$ denote solutions to \eqref{eq:ave} and \eqref{eq:ave_sys} respectively. Averaging methods are then used to make conclusions about $x(t,\epsilon)$ based on solutions to the averaged (and easier to analyse) system \eqref{eq:ave_sys} using the fact that 
\begin{equation}\label{eq:order}
x(t,\epsilon)-\widetilde{x}(t\epsilon)= \cO(\delta(\epsilon)) \text{ on } [0,\alpha], 
\end{equation}
if the equilibrium point at the origin of the averaged system is exponentially stable. The time interval defined by the positive constant $\alpha$ varies depending on various stability assumptions, the function $\delta$ is of class $\cK$ and depends on the averaging function.\footnote{A function $\delta:[0,a]\rightarrow [0,\infty)$ with $a>0$ is a class $\cK$ function if $\delta$ is strictly increasing and $\delta(0)=0$. } In order for the order bound \eqref{eq:order} to hold it is assumed that the Jacobian of
\begin{equation*}
g(t,x):= f(t,x,0)-\widetilde{f}(x) 
\end{equation*}
has zero average and the same convergence function as $f$.

Averaging methods in their standard form provide a method for reasoning about systems of the form \eqref{eq:ave} by analysing their averaged counterpart \eqref{eq:ave_sys} which is often of a simpler form than the original system. However, implicit in this is the assumption that the state equation depend \emph{smoothly} on $\epsilon$. When such a smoothness assumption fails, provided the system can be written in the form
\begin{subequations}
\begin{align}
\dot x_1&=f_1(t,x_1,x_2,\epsilon), \label{eq:slow1} \\
\dot x_2 &= \frac{1}{\epsilon}f_2(t,x_1,x_2,\epsilon), \label{eq:fast1}
\end{align}
\end{subequations}
then averaging can be used to integrate out the variable $x_2$ from the dynamics and produce bounds of the form of \eqref{eq:order}. In essence good approximations of the slow dynamics \eqref{eq:slow1} can be made by averaging the fast dynamics \eqref{eq:fast1}.

\subsection{Stochastic Averaging for Time-Scale Separation \label{ss:stoch-aver}}
The stochastic process $X(t)$ is said to be second-order if $\E X^2(t) < \infty$ for all $t$.  We will assume throughout that all stochastic processes satisfy the second-order condition. A Weiner process $w(t)$ is a stochastic process that satisfies the following conditions: (1) $w(0)=0$, (2) $w(t)-w(s)\sim \cN(0,(t-s) I)$ for any $0 \le s < t$ and increments of non-overlapping time intervals are independent, (3) $\E w(t)=0$ for all $t>0$. We consider stochastic differential equations (SDEs) in the following form:
\begin{equation}\label{eq:SDE}
\dot X = b(t,X) + \sigma(t,X) \dot w
\end{equation}
with $X(t)\in \R^n$, and $w(t)$ an $n$-dimensional Weiner process. From a stochastic calculus point of view the derivative of a Weiner process is a notation for the Brownian motion and hence our SDEs are in the sense of It\^o. We also make a standing assumption that $b(t, X)$ and $\sigma(t,X)$ are measurable, bounded and globally Lipschitz in $X$ uniformly over $t$, which implies that the solutions to~\eqref{eq:SDE} exist and are unique (cf.~\cite{Reiss2003}). 

Let us now consider the averaging principle for processes defined by stochastic differential systems. Following the presentation of \cite{freidlin2012random}, consider the system
\begin{subequations}
\begin{align}
\dot{X} &= b(X,Y)+\sigma(X,Y)\dot{w}, \quad X(0)=x \label{eq:slows1}\\
\dot{Y}& = \epsilon^{-1}B(X,Y)+\epsilon^{-1/2}C(X,Y)\dot{w}, \quad Y(0)=y \label{eq:slows2}
\end{align}
\end{subequations}
where $w$ is an $s$-dimensional Weiner process, $X \in \R^{n_1}$, $Y \in \R^{n_2}$, $\sigma(x,y)$ and $C(x,y)$ are matrices that map $\R^s$ to $\R^{n_1}$ and $\R^{n_2}$ respectively. To simplify things we assume that $\sigma$ only depends on $x$, i.e. $\sigma(x,y)=\sigma(x)$. 

Assume there exists a function $\widetilde{b}(x)$, such that for any $t\ge0$, $x\in \R^{n_1}$ and $y\in \R^{n_2}$ we have
\begin{equation}\label{eq:ave_stochastic}
\E \left\|\frac{1}{T}\int_t^{t+T}b(x,Y^{xy})ds-\widetilde{b}(x) \right\| < \chi(T)
\end{equation}
where $\chi$ is a convergence function and $Y^{xy}$ is a stochastic process defined  by
\begin{equation*}
\dot{Y}^{xy}(t) = B(x,Y^{xy}(t))+C(x,Y^{xy}(t))\dot{w}.
\end{equation*}
The following result summarises the classical stochastic averaging framework.

\begin{prop}[\cite{freidlin2012random}] 
Let $X(t)$ denote the random process that satisfies~\eqref{eq:slows1},~\eqref{eq:slows2} and assume that~\eqref{eq:ave_stochastic} holds. Define $\widetilde{X}(t)$ to be the process determined by
\begin{equation*}
\dot{\widetilde{X}}= \widetilde{b}(\widetilde{X}(t)) + \sigma(\widetilde{X}(t))\dot{w}(t), \quad \widetilde{X}(0)=x.
\end{equation*}
Then for any $T>0$, $x\in \R^{n_1}$ and $y\in \R^{n_2}$, we have
\begin{equation*}
\lim_{\epsilon \rightarrow 0}P \left\{ \sup_{0\le t \le T}\| X(t)-\widetilde{X}(t) \|>\delta \right\}=0.
\end{equation*}
\end{prop}

If $\sigma$ depends on both $X$ and $Y$, then a similar averaging procedure can be derived. However, the convergence result is weaker, namely, the process $X(t)$ converges weakly (in distribution) to the process $\widetilde{X}(t)$ with $\varepsilon\rightarrow 0$~\cite{khas1968aver_c}.

Note that when  $\sigma$ and $C$ equal zero then we have an ordinary differential equation and can obtain a generalisation of the celebrated Tikhonov theorem (cf.~\cite{kokotovic1999singular}). The main difference between the two approaches is that in Tikhonov theorem one sets $\varepsilon$ to zero in~\eqref{eq:slows2} and solves $B(X, Y) = 0$  for $Y$. In the stochastic case, this procedure cannot be applied, since $\dot w$ has infinite variation. Hence integrating variables out is essential for time-scale separation of stochastic processes. 
Recent work towards removing some of the strict assumptions in stochastic averaging, such as global Lipschitzness, equilibrium preservation, and exponential stability has been reported in~\cite{LiuK10,LiuK10a}.

\section{Averaging of the Linear Noise Approximation\label{s:aver-lna}}
In this paper, we consider an approximation of the CME; the \emph{Linear Noise Approximation} or LNA (cf.~\cite{thomas2012rigorous}). The LNA constitutes a valid approximation of the CME if a large number of reactions occur per unit time and additionally the volume $\Omega$ is sufficiently large. In this case let
\[
\frac{n}{\Omega} = x + \Omega^{-1/2} \eta,
\]
where $x$ is the vector of macroscopic concentrations of the species and $\eta$ is a vector of stochastic fluctuations about $x$. Now by applying the Taylor expansion to the CME, it can be shown (cf.~\cite{thomas2012rigorous}), that the fluctuations $\eta$ and macroscopic concentrations $x$ obey the following equations: 
\begin{subequations}\label{eq:LNA}
\begin{align}
\label{eq:det-dyn}   &\dot x = g(x), \\
\label{eq:lna-fla}   &\dot \eta = A(x) \eta + B(x) \dot w, 
\end{align}
\end{subequations}
where $g(x) = S f(x)$, $A(x)$ is the Jacobian of $S f(x)$, $B(x) = \Omega^{-1/2} S \diag{\sqrt{f(x)}}$, $S$ is the stoichiometric matrix from \eqref{eq:CME}, $f(x)$ is an approximation of $\hat f(n,\Omega)$ and $w$ is a $R$-dimensional Weiner process. Note that the \emph{macroscopic} fluctuation $f$ approximating the microscopic rate functions $\hat{f}$ for the four fundamental reactions as well as some more complex reactions are given in \cite{ThoSG12}. In fact it is shown that $\lim_{\Omega \rightarrow \infty}\hat{f}(n,\Omega)=f(n)$. 

To streamline presentation we will drop the dependence on $x$ from the notation when referring to $A(x)$ and $B(x)$. An important observation is that the the matrices $A$, $B$ do not depend on the fluctuations $\eta$, but depend only on the macroscopic concentrations $x$, which is computed using deterministic differential equations. Therefore the fluctuation dynamics constitute a linear time-varying SDE and the mean and the covariance of $\eta$ can be computed as follows (cf.~\cite{freidlin2012random}).

\begin{prop}\label{prop:cov}
The covariance $\cov(\eta(t)) = P(t)$ and the mean $\E(\eta(t)) =m(t)$ of the solution to the SDE~\eqref{eq:lna-fla} satisfy
\begin{align}
\label{eq:mean-lna}\frac{dm}{dt} & =A m,\quad m(t_0) = m_0,\\ 
\label{eq:cov-lna} \frac{dP}{dt} &= A P+P A^\ast + B B^\ast, \quad P(t_0)=P_0,
\end{align}
where $X(t_0)\sim \cN(m_0, P_0)$.
\end{prop}
%

We start the discussion of  time-scale separation of the LNA by assuming that the state vector of~\eqref{eq:LNA} has been appropriately permuted and partitioned as:
\begin{equation}\label{eq:partition}
\begin{bmatrix} x^\ast & \eta^\ast \end{bmatrix}^\ast = \begin{bmatrix} x_1^\ast & x_2^\ast & \eta_1^\ast & \eta_2^\ast \end{bmatrix}^\ast.
\end{equation}
The vector field $g$ and the matrices $A$ and $B$ can then be conformally partitioned according to~\eqref{eq:partition}. Note that the \emph{true state} of the system is given by $x + \eta$. Assume that $x_1 + \eta_1$ varies on the time scale, which is $\varepsilon$ times slower than the time scale of $x_2 + \eta_2$. In this case the LNA~\eqref{eq:LNA} can be written as follows~\cite{thomas2012rigorous}:
\begin{subequations}\label{eq:full-model}
\begin{eqnarray}
\dot x_1 &=& g_1(x_1, x_2), \label{eq:full-model-x1}\\
\dot x_2 &=& \varepsilon^{-1} g_2(x_1, x_2), \label{eq:full-model-x2}\\
\dot \eta_1 &=& A_{1 1} \eta_1 + \varepsilon^{-1/2} A_{1 2} \eta_2+ 
\label{eq:full-model-eta1}  B_1 \dot w, \\
\dot \eta_2 &=& \varepsilon^{-1/2} A_{2 1} \eta_1 + \varepsilon^{-1} A_{2 2}\eta_2 +
\varepsilon^{-1/2} B_2 \dot w. \label{eq:full-model-eta2}
\end{eqnarray}   
\end{subequations}

Under standard conditions on time-scale separation, the reduced order model is given by the following equations:
\begin{subequations}\label{red-model}
\begin{align}
&\dot z   = g_1(z, \hat z), \label{full-model-x1}\\
&\dot \xi = A_r(z, \hat z)\xi+ B_r(z, \hat z) \dot w, \label{red-model-xi}
\end{align}   
\end{subequations}
where $\hat z$ is the unique root of the equation  $g_2(x_1, x_2) = 0$ solved with respect to $x_2$, and 
\begin{equation}\label{def:ab}
\begin{gathered}
A_r = A_{1 1}- A_{1 2} A_{2 2}^{-1} A_{2 1}, \\
B_r = B_{1}- A_{1 2} A_{2 2}^{-1}B_{2}.
\end{gathered}
\end{equation}
 We can now present the first main result of the paper:
\begin{theorem} \label{thm:conv} Consider the system~\eqref{eq:full-model}, where $g_1(x)$, $g_2(x)$ are continuously differentiable functions with bounded derivatives and $A_{2 2}$ is invertible along the trajectory of the full order model~\eqref{eq:full-model} and locally exponentially stable for all $x$. Let the system~(\ref{eq:full-model-x1}-\ref{eq:full-model-x2}) satisfy standard assumptions on time-scale separation in~\cite[pp. 9--11]{kokotovic1999singular}. Then there exists $\varepsilon_1$ such that for all $\varepsilon$ satisfying $\varepsilon_1 \ge \varepsilon \ge 0$ we have
\[
\sup\limits_{0 \le t \le T}\left\|
z - x_1\right\|_2 = O(\varepsilon)~~
\sup\limits_{0 \le t \le T}\E\left\|
\xi - \eta_1 \right\|_2^2 = O(\varepsilon) 
\]
where $z$, $\xi$ are solutions to~\eqref{red-model}.
\end{theorem}

Before providing the proof we note that this reduced order model was derived in~\cite{thomas2012rigorous}, however, it was only argued that there is convergence in distribution with $\varepsilon\rightarrow 0$. This method can be seen as a type of stochastic averaging, since fast variables $x_2+\eta_2$ are essentially integrated out. We also note that a similar convergence result to the reduced order model (\ref{full-model-x1},\ref{red-model-xi}) can be shown if the fluctuations $\eta_1$, $\eta_2$ evolve according to the following model
\begin{gather*}
\dot{\tilde \eta}_1 = A_{1 1} \tilde \eta_1 +  A_{1 2} \tilde \eta_2 + B_1 \dot w, \\
\dot{\tilde \eta}_2 = \varepsilon^{-1} (A_{2 1} \tilde \eta_1 + A_{2 2} \tilde \eta_2 + B_2 \dot w).
\end{gather*}
This model can be obtained from~(\ref{eq:full-model-eta1},\ref{eq:full-model-eta2}) by a change of variables $\tilde \eta_1 = \eta_1$, $\tilde \eta_2 = \varepsilon^{-1/2} \eta_2$. The following lemma is required before we can state the proof:
\begin{lem}\label{lem:lip-bounds} The functions $A_r(\cdot)$, $B_r(\cdot)$ defined in~\eqref{def:ab} satisfy the following bounds 
\begin{align}
\notag \int_0^t\|A_r(z, \hat z) - A_r(x_1, x_2)\|_2^2 d\tau \le O(\varepsilon), \\
\notag \int_0^t\|B_r(z, \hat z)-B_r(x_1, x_2)\|_2^2 d\tau \le O(\varepsilon), \\
\notag \|A_r(x_1, x_2)\|_2^2 \le K_{1},
\end{align}
where $x_1(t)$, $x_2(t)$, $z(t)$, $\hat z$ are defined in~\eqref{eq:full-model}, and~\eqref{red-model}.
\end{lem}

\begin{proof}
According to~\cite{kokotovic1999singular} the root $\hat z$ exists and for all $t$ such that $0\le t\le T$ we have that
\begin{equation}
\label{bounds:tss}
\begin{aligned}
x_1(t) - z(t)= O(\varepsilon),
\end{aligned}
\end{equation}
therefore the statement is well-posed. 

We begin with the final inequality in the lemma. As the elements of the matrix $A_r$ are polynomials it follows from the definition on the operator norm and the fact that polynomial functions are Lipschitz on a compact domain of arbitrary size c.f. Lemma~\ref{lem:lip-bounds-proof} in Appendix \ref{app:convergence}. Using this bound and the equivalence of norm property, we can show the second inequality as follows
\begin{multline*}
\notag \int_0^t\|A_r(z, \hat z) - A_r(x_1, x_2)\|_2^2 d\tau \\
\le \int_0^t\|A_r(z, \hat z) - A_r(x_1, x_2)\|_{F}^2 d\tau \\
\le L_1 \int_0^t(\|x_1(t) - z(t)\|_2^2 + \|x_2(t) - \hat z(t)\|_2^2) d\tau,
\end{multline*}
where $\|\cdot\|_F$ is the Frobenius matrix norm.

Furthermore, $x_2(t)$ asymptotically converges to $\hat z(t)$ with $\varepsilon \rightarrow 0$, hence there exists a small enough $\varepsilon_1$ such that for all $0 \le \varepsilon \le \varepsilon_1$ we have $\int_0^t\|x_2(t) - \hat z(t)\|_2^2 d\tau \le O(\varepsilon)$. The same argument is used for the inequality involving $B_r$. 
\end{proof}
We are now ready to present the proof of Theorem \ref{thm:conv}.

\begin{proof-of}{Theorem~\ref{thm:conv}} 
According to \cite{kokotovic1999singular} the root $\hat z$ exists and for all $t$ such that $0\le t\le T$ we have that $x_1(t) - z(t)= O(\varepsilon)$, therefore we only need to prove convergence of the fluctuation dynamics. We can rewrite the equation for the slow perturbation variable as follows:
\begin{align*}
 \dot{\eta}_1 &= A_{1 1} \eta_1 + \varepsilon^{-1/2} A_{1 2}\eta_2 + B_1 \dot w \\&= A_r(x_1, x_2) \eta_1 + B_r(x_1,x_2) \dot w\\
    &\quad+ A_{1 2} A_{2 2}^{-1}(\varepsilon^{-1/2} A_{2 2} \eta_2 + A_{2 1} \eta_1 + B_2 \dot w)
\end{align*}

Taking into account this representation we obtain
\begin{multline}
\label{eq:err-st}
\xi(t)- \eta_1(t) = \int_0^{t} (B_r(z, \hat z) - B_r(x_1, x_2)) \dot w  d\tau + \\
\int_0^{t} \left(A_r(z, \hat z)- A_r(x_1, x_2)\right)\xi  d\tau + \int_0^{t} A_r(x_1, x_2) \left(\xi -  \eta_1 \right) d\tau \\
+\underbrace{\int_{0}^{t} A_{1 2} A_{2 2}^{-1}(\varepsilon^{-1/2} A_{2 2} \eta_2   + A_{2 1} \eta_1 + B_2 \dot w) d \tau}_{C_1(t)},
\end{multline}
where the matrices $A_{i j}$, $B_i$ depend on $x(\tau)$. We prove the main result by showing that the expectation of each of the terms on the right hand side of \eqref{eq:err-st} are of order $O(\varepsilon)$. Due to the It\^o isometry rule (cf.~\cite{Reiss2003}) we have  
\begin{multline*}
 \E \left\|\int_0^{t} (B_r(z, \hat z)   - B_r(x_1, x_2)) \dot w  d\tau\right\|_2^2 \\
 = \int_0^{t} \left\|B_r(z, \hat z) - B_r(x_1, x_2)\right\|_F^2  d\tau  \\
 \le L_1 \int_0^{t} \|B_r(z, \hat z) - B_r(x_1, x_2)\|_2^2 d\tau \le O(\varepsilon),
\end{multline*}
where the last inequality is due to Lemma~\ref{lem:lip-bounds}.  By using consecutively Jensen, Cauchy-Schwartz inequalities, and  the bounds in Lemma~\ref{lem:lip-bounds} we have
\begin{multline*}
\E\left\|\int_0^{t} \left(A_r(z, \hat z)-A_r(x_1, x_2)\right)\xi  d\tau\right\|_2^2  \\ \le \int_0^{t} \E\|\left(A_r(z, \hat z)-A_r(x_1, x_2)\right)\xi\|_2^2  d\tau  \\ \le
L_2 \int_0^{t} \left\|A_r(z, \hat z)-A_r(x_1, x_2)\right\|_F^2 d\tau \int_0^{t} \E \xi^2  d\tau \le O(\varepsilon),
\end{multline*}
Similarly we can show that 
\begin{gather*}
\E\left\|\int_0^{t} A_r(x_1, x_2) (\eta_1 - \xi ) d\tau\right\|_2^2 \le K_1 \int_0^t \E \|\xi(t) -  \eta_1(t)\|^2 d\tau
\end{gather*}
for some positive $K_1$. Using Lemma \ref{lem:c1-bounds} (see Appendix \ref{app:convergence}) it is shown that $E\|C_1(t)\|^2  = O(\varepsilon)$. Finally, 
let $m_\varepsilon(t) = \E \|\xi(t) -  \eta_1(t)\|_2^2$, by applying the previous bounds to~\eqref{eq:err-st}, we obtain
\[
m_\varepsilon(t) \le K_1 \int_{0}^{t} m_\varepsilon(\tau) d \tau + O(\varepsilon).
\]
Therefore by Lemma~\ref{prop:gronwall} (see Appendix~\ref{app:convergence}) we have that 
\[
m_\varepsilon(t) \le O(\varepsilon) e^{K_1 t}.
\]
As the above inequality is considered only on a finite interval it follows that $m_\varepsilon(t) = O(\varepsilon)$ for all $t\in[0, T]$. 
\end{proof-of}

\section{Solutions to Structured Model Reduction by Balancing\label{s:mor}}
We now turn our attention to structured model reduction. A method for computing structure preserving reduced order models is presented as well as some remarks on existence of structured solutions. Then in Section \ref{s:proj-lna} using insight from the previous section we apply our algorithm to the LNA approximation of the CME. 
\subsection{(Structured) Model Order Reduction by Balancing\label{ss:smor}}
Assume that we have a stable system with $n_i$ inputs, $n_o$ outputs, $n$ states. We will adopt the following shorthand notation for a realisation of a system:
\begin{equation}\label{eq:original_real}
\cG = \left[\begin{array}{c|c} A & B \\ \hline\\[-9pt]
 C & D \end{array} \right] ~ \Leftrightarrow ~\cG(s)=C(sI-A)^{-1}B +D.
\end{equation}

The essential step in projection-based methods is naturally the computation of the so-called projectors $V \in \R^{k\times n}$ and $W \in \R^{k\times n}$, where $k<n$ is a pre-defined order of the reduced model. Given the projectors the approximate  model  is computed as follows:
\begin{equation}
\cG^k = \left[\begin{array}{c|c} V A W & V B \\ \hline\\[-9pt] C W & D \end{array} \right], \label{real:proj-based-mor}
\end{equation}
where the superscript indicates the state dimension. In particular we would like to minimise 
\[\|\cG - \cG^k\|_{\Hinf},
\] 
where \begin{align*}
\|\cG\|_{\Hinf}:= \sup_{\text{Re}(s)>0} \bar{\sigma}[\cG(s)] = \text{ess }\sup_{\omega \in \R}\bar{\sigma}[\cG(j\omega)]. 
\end{align*}

The projectors $V$ and $W$ can be computed using interpolation methods based on Krylov subspace techniques (cf.~\cite{AntoulasBook}). However, here we will employ balancing tools (cf.~\cite{AntoulasBook}) because in this case we can compute \emph{structured projectors} introduced in the sequel. 

First we cover the celebrated balanced truncation method and consider the Lyapunov equations
\begin{subequations}\label{eq:lyap}
\begin{eqnarray}
A P +P A^\ast + B B^\ast &=&0\\
A^\ast Q+ Q A + C^\ast C &=&0
\end{eqnarray}   
\end{subequations}
which are associated with the realisation of $\cG$. If $A$ is asymptotically stable then there exist unique solutions $P\succeq 0$, $Q\succeq0$ to~\eqref{eq:lyap}, which are called controllability and observability Gramians, respectively. If additionally $(A, B)$ is controllable and/or $(A,C)$ is observable, then the respective Gramian will be positive definite. The eigenvalues $\sigma_i$ of the matrix $(P Q)^{1/2}$ are referred to as \emph{Hankel singular values}. We call a realisation \emph{balanced}, if $P = Q = \Sigma = \diag{\sigma_1,\dots,\sigma_n}$. The following proposition summarises two key balancing methods for model reduction.

\begin{prop}\label{prop:br} Consider a realisation of $\cG$ from~\eqref{eq:original_real} and assume there exist positive definite matrices $P$ and $Q$ satisfying~\eqref{eq:lyap}. Let $\sigma_i$ be the Hankel singular values of $\cG$, let $k$ be such that $\sigma_i \ne \sigma_j$ for all $i\le k$, $j>k$. Let $T$ be such that  $T P T^\ast = (T^\ast)^{-1} Q T^{-1}=\Sigma = \diag{\sigma_1,\dots, \sigma_n}$ and consider the realisation with the following partitioning:
\begin{equation}\label{eq:balanced_real}
\hat \cG =  \left[\begin{array}{c|c} T A T^{-1} & T B \\ \hline\\[-9pt] C T^{-1}& D \end{array} \right] =\left[\begin{array}{c c|c} \hat A_{1 1} & \hat A_{1 2} & \hat B_1 \\
                                       \hat A_{2 1} & \hat A_{2 2} & \hat B_2 \\
 \hline\\[-9pt] 
 \hat C_1 & \hat C_2 & D \end{array} \right],
\end{equation}
where $\hat A_{1 1}\in \R^{k\times k}$, $\hat B_{1}\in \R^{k\times n_i}$, $\hat C_{1}\in \R^{n_o\times k}$. Then $\hat \cG$ is balanced and the reduced order realisations $\hat \cG_1^k$, $\hat \cG_2^k$ defined as
\begin{gather*}
\hat \cG_1^k =  \left[\begin{array}{c|c} \hat A_{1 1} & \hat B_1 \\ \hline\\[-9pt] \hat C_1 & D \end{array} \right], \\
\hat \cG_2^k =  \left[\begin{array}{c|c} \hat A_{1 1} - \hat A_{1 2} \hat A_{2 2}^{-1} \hat A_{2 1} & \hat B_1 -\hat A_{1 2}\hat A_{1 2}^{-1} \hat B_{2}   \\ \hline\\[-9pt] \hat C_1 -\hat C_{2}\hat A_{2 2}^{-1} \hat A_{2 1}  & D-\hat C_{2}\hat A_{2 2}^{-1} \hat B_{2}  \end{array} \right] .
\end{gather*}
are both asymptotically stable, balanced,  and satisfy the error bounds:
\begin{gather*}
\|\cG - \hat \cG_1^k\|_{\Hinf}  \le 2\sum\limits_{i = k+1}^n \sigma_{i},\quad
\|\cG - \hat \cG_2^k\|_{\Hinf} \le 2\sum\limits_{i = k+1}^n \sigma_{i}.
\end{gather*}
\end{prop}

The application of a transformation $T$ to $\cG$ can be seen as a transformation of the state-space variable $\hat{x}=T x$. The balancing transformation is computed as $T = \Sigma^{1/2} U R^{-1}$, where $R$ is a lower triangular matrix such that $P = R^\ast R$, while $\Sigma$ and $U$ are obtained from the singular value decomposition $R Q R^\ast = U\Sigma^2 U^*$. Note that we do not require $\sigma_i \ge \sigma_{i+1}$ for all $i =1,\dots,n-1$, which may be confusing in the standard formulation but is useful in the structured extension. The proof for the error bounds still holds in this case.

The algorithm to compute $\hat \cG_1^k$ is due to Moore~\cite{Moore1981} and is usually called \emph{balanced truncation}, the algorithm to compute $\hat \cG_2^k$ is due to Liu and Anderson~\cite{liu1989singular} and is called \emph{balanced singular perturbation}. The balanced singular perturbation approach matches the full order system at the zero frequency point, while the balanced truncation matches the full system in the frequency equal to infinity. Note also that balanced singular perturbation can be derived as the balanced truncation of a system with $\widetilde A = A^{-1}$, $\widetilde B = A^{-1} B$, $\widetilde C = C A^{-1}$, $\widetilde D = D - C A^{-1} B$. 
Finally, the balanced truncation is called projection-based since we can set $V = \begin{pmatrix} I_k & 0 \end{pmatrix} T$ and $W^\ast = \begin{pmatrix} I_k & 0 \end{pmatrix} (T^\ast)^{-1}$ such that the reduced order model $\cG^k$ has the realisation~\eqref{real:proj-based-mor}.

The balancing transformation $T$ is typically a full matrix, and hence any physical meaning in $x$ is not preserved in the new variables $\hat x = T x$. Additionally the sparsity of the drift matrix is lost under the transformation $TAT^{-1}$. In some cases, the state $x$ is naturally partitioned into two groups of states $x_1$, $x_2$ such that $x = \begin{pmatrix}x_1^\ast & x_2^\ast \end{pmatrix}^\ast$, as for example in closed loop systems with a controlled system having the states $x_1$ and a controller with the states $x_2$. In this case, it is not desirable for the transformation $T$ to mix $x_1$ and $x_2$ as to do so would destroy the controlled system-controller structure. Hence, $T$ has to be block-diagonal, and consequently so should the Gramians $P$ and $Q$. In order to find block-diagonal Gramians  we need to consider the so-called Lyapunov inequalities instead of equations:
\begin{subequations}\label{eq:lyap-ineq}
\begin{eqnarray}
A P +P A^\ast + B B^\ast &\prec&0,\\
A^\ast Q+ Q A + C^\ast C &\prec&0.
\end{eqnarray}   
\end{subequations}
The solutions $P$ and $Q$ to~\eqref{eq:lyap-ineq} are called \emph{generalised Gramians}. In what follows, we will consider generalised Gramians with a certain sparsity pattern $\cS$, and write $P\in\cS$, if $P$ has the sparsity pattern $\cS$. Note that these Gramians may not exist in general, but we discuss their existence in the next subsection. If $P = Q = \Sigma = \diag{\sigma_1, \dots, \sigma_n}$, then the realisation is called \emph{balanced in the generalised sense}, while $\sigma_i$'s \emph{generalised Hankel singular values}. The model reduction procedure is the same, and the error bounds are given in the form of generalised Hankel singular values. We refer to the reduction procedure, which consists of computing structured Gramians and projections as \emph{structured model reduction}. The following result generalises Proposition~\ref{prop:br} to structured model reduction. The error bounds for the structured balanced truncation are first shown in~\cite{Sandberg09}, however, we find our proof simpler and more intuitive. Moreover, to our best knowledge the proof for structured balanced singular perturbation is novel. 

\begin{theorem}\label{prop:gen_gram}  Consider a realisation $\cG$ in~\eqref{eq:original_real} and let there exist positive definite $P\in\cS$ and $Q\in\cS$ satisfying inequalities in~\eqref{eq:lyap-ineq}. Then the statement of Proposition~\ref{prop:br} holds, while the Hankel singular values and balancing are understood in the generalised sense.
\end{theorem}
\begin{proof} 
(i) \emph{Error bounds for structured balanced truncation.} Let there exist $P\in\cS$ and $Q\in\cS$ satisfying inequalities~\eqref{eq:lyap-ineq}. Then the following equations hold for some positive definite matrices $X$ and $Y$:
\begin{subequations}\label{eq:lyap-ineq-eq}
\begin{eqnarray}
A P +P A^\ast + B B^\ast + X &=&0,\\
A^\ast Q + Q A + C^\ast C + Y &=&0.
\end{eqnarray}   
\end{subequations}
Due to positive definitiveness of $X$ and $Y$, there exist such $B_e$ that $X = B_e B_e^\ast$, and such $C_e$ that $Y = C_e^\ast C_e$. Hence we can treat $B_e$ as another input matrix, and $C_e$ as another output matrix for some extended transfer function $\cG_e$
\begin{equation*}
\cG_e = \left[\begin{array}{c|c c}  A &  B &  B_e \\ \hline\\[-9pt]
 C & D & 0  \\
 C_e & 0 & 0 \end{array} \right],
\end{equation*}
which we can balance with respect to Gramians $P$, $Q$ from~\eqref{eq:lyap-ineq-eq} and partition as follows
\begin{equation*}
\hat{\cG}_e = \left[\begin{array}{c c|c c} \hat A_{1 1} & \hat A_{1 2} & \hat B_1 & \hat B_{e1} \\ 
                                   \hat A_{2 1} & \hat A_{2 2} & \hat B_2 & \hat B_{e2} \\ 
\hline\\[-9pt]
\hat C_1 & \hat C_2 & D & 0  \\
\hat C_{e1} & \hat C_{e2} & 0 & 0  \end{array} \right].
\end{equation*}
Consider the reduced order model $\hat \cG_{e}^k$
\begin{gather*}
\hat \cG_{e1}^k =  \left[\begin{array}{c|c c} \hat A_{1 1} & \hat B_1 & \hat B_{e 1} \\ \hline\\[-9pt] 
                                          \hat C_1 & D & 0\\
                                          \hat C_{e 1} & 0 & 0\end{array} \right], 
\end{gather*}
which according to Proposition~\ref{prop:br} is asymptotically stable, balanced, and fulfils the following error bound:
\begin{gather*}
 \|\cG_{e} - \hat \cG_{e1}^k\|_{\Hinf} \le 2\sum\limits_{i = k+1}^n \sigma_{i}.
\end{gather*}
Finally note that the upper left corner of the transfer function $\cG_e(s)$ (respectively, $\hat \cG_{e1}^k$) is equal to $\cG(s)$ (respectively, $\hat \cG_{1}^k$). This implies that
\begin{gather*}
\|\cG - \hat \cG_{1}^k\|_{\Hinf} \le \|\cG_{e} - \hat \cG_{e1}^k\|_{\Hinf}\le 2\sum\limits_{i = k+1}^n \sigma_{i}
\end{gather*}
and proves the claim. 

(ii) \emph{Error bounds for  balanced singular perturbation.} We repeat the derivation above and obtain the reduced model $\hat \cG_{e2}^k$
\begin{gather*}
\hat \cG_{e2}^k =  
\left[\begin{array}{c|cc} \hat A_{1 1} - \hat A_{1 2} \hat A_{2 2}^{-1} \hat A_{2 1} & \hat B_1 -\hat A_{1 2}\hat A_{1 2}^{-1} \hat B_{2}  & \tilde B_2  \\ \hline\\[-9pt] 
\hat C_1 -\hat C_{2}\hat A_{2 2}^{-1}\hat A_{2 1} & D -\hat C_2 \hat A_{2 2}^{-1} \hat B_2 & \tilde D_{1 2}\\
\hat C_{e1} -\hat C_{e2}\hat A_{2 2}^{-1}\hat A_{2 1} & -\hat C_{e 2} \hat A_{2 2}^{-1} \hat B_2 & \tilde D_{2 2}\end{array} \right], 
\end{gather*}
where $\tilde B_2 =\hat B_{e1} -\hat A_{1 2}\hat A_{1 2}^{-1} \hat B_{e2}$, $\tilde D_{1 2} = -\hat C_2 \hat A_{2 2}^{-1} \hat B_{e 2}$, $\tilde D_{2 2} = -\hat C_{e 2} \hat A_{2 2}^{-1} \hat B_{e 2}$. According to Proposition~\ref{prop:br} $\cG^k_2$ is asymptotically stable, balanced, and fulfils the following bound:
\begin{gather*}
 \|\cG_{e} - \hat \cG_{e2}^k\|_{\Hinf} \le 2\sum\limits_{i = k+1}^n \sigma_{i}.
\end{gather*}
Again the upper left corner of the transfer function $\cG_e(s)$ (respectively, $\hat \cG_{e2}^k$) is equal to $\cG(s)$ (respectively, $\hat \cG_{2}^k$). Hence: 
\begin{gather*}
\|\cG - \hat \cG_{2}^k\|_{\Hinf} \le \|\cG_{e} - \hat \cG_{e2}^k\|_{\Hinf}\le 2\sum\limits_{i = k+1}^n \sigma_{i}.
\end{gather*}
which completes the proof.
\end{proof}

Since there are infinitely many solutions to the matrix inequalities~\eqref{eq:lyap-ineq}, there are infinitely many combinations of the generalised Hankel singular values $\{\sigma_1, \dots, \sigma_n\}$. We want, however, to find such $P$ and $Q$ that the smallest $\sigma_k$ are close to zero in order to achieve minimal reduction error. In this case, the following heuristic is usually proposed.
\begin{gather} \label{lyap-ineq-smor}
\begin{aligned}
\min\limits_{P\in \cS, P \succ 0}~~~&\tr(P) \\
\text{subject to: } &A P + P A^\ast + B B^\ast \prec 0,
\end{aligned}
\end{gather}
where $\cS$ is a sparsity constraint on $P$, e.g., $P$ is block-diagonal. The programme to compute $Q$ is derived in a similar manner. The trace minimisation here acts as a rank minimisation programme on $P$ and $Q$, thus minimising the smallest generalised Hankel singular values.

\subsection{Algebraic Conditions for Existence of Diagonal Generalised Gramians}\label{ss:exist}
As described above the block-diagonal transformation $T$ exists, if there exist generalised Gramians with the same sparsity pattern. The question of existence of such Gramians is studied in, for example,~\cite{sootla2016existence, andersondecentralised} and the references therein. 
In preparation of this manuscript we have noticed that many biochemical networks, when linearised around a steady-state have diagonal Gramians. Therefore we focus on \emph{diagonally stable} drift matrices $A$, meaning that there exists a positive definite, diagonal $P$ such that $A P + P A^\ast$ is negative definite. If the drift matrix $A$ in the realisation $\cG$ is diagonally stable, then we can find diagonal \emph{generalised Gramians} as shown in~\cite{sootla2014projectionI}. There exist no easy parametrisations of the class of diagonally stable matrices. However, there are some sufficient conditions for diagonal stability. For example in~\cite{arcak2011diagonal}, it was shown that for the so-called \emph{cacti graphs} there exist necessary and sufficient conditions for diagonal stability. We will not formally define these graphs, but just mention that these graphs rule out sparsity patterns generated by reversible reactions. Our results build upon those presented in~\cite{hershkowitz1985lyapunov} that certify existence of diagonal Lyapunov functions for a broad class of graphs which seem to appear frequently in biochemical reaction networks. To proceed we require a few definitions:
\begin{defn}
A matrix $A=\{a_{i j}\}_{i,j = 1}^n \in\R^{n\times n}$ is called Metzler, if it has nonnegative off-diagonal elements, that is $a_{i j} \ge 0$ for $i\ne j$.
\end{defn}
\begin{defn}
A matrix $\cM(A) = \{ m_{i j}\}_{i, j = 1}^n$ is called a companion matrix of $A = \{a_{i j}\}_{i,j = 1}^n \in\R^{n\times n}$, if 
\begin{equation*}
m_{i j} = \left\{\begin{array}{ll} |a_{ i j}| &  i = j \\
- |a_{i j}| & i \ne j
\end{array}\right.
\end{equation*}
\end{defn}
\begin{defn}
A matrix $A \in\R^{n\times n}$ is called an $\mH$ matrix if $\cM(A)$ has all eigenvalues with a nonnegative real part. If additionally $a_{ i i} >0$ for all $i$ we say that $A$ is an $\mH_+$ matrix.
\end{defn}
\begin{defn}
A matrix $A \in\R^{n\times n}$ is called strictly row scaled diagonally dominant if there exist positive $d_1, \dots, d_n$ such that for all $i=1,\dots,n$ we have
\[
d_i |a_{i i}| > \sum\limits_{j \ne i} d_j |a_{i j}|.
\]
The matrix $A$ is called strictly column scaled diagonally dominant, if $A^\ast$ is strictly row scaled diagonally dominant. If we can choose $d_i = 1$ for all $i$, then the matrix $A$ is called strictly row diagonally dominant, and we write $A\in \cDD$. If additionally $a_{i i} >0$ for all $i$, then we write $A\in \cDDp$.
\end{defn}

It can be shown that the class of $\mH$ matrices contains stable Metzler and stable triangular matrices; additionally, nonsingular matrices are $\mHp$ matrices if and only if they are strictly row and column scaled diagonally dominant~\cite{varga1976recurring}. Moreover, symmetric $\mHp$ matrices are positive definite, which can be shown by virtue of \emph{Gershgorin circle theorem}; however the reverse implication does generally hold. 

\begin{prop}[\cite{hershkowitz1985lyapunov}] \label{prop:h-diag-stab}
Let $-A$ be an $\mHp$ matrix. Then $A$ is diagonally stable if and only if $A$ is nonsingular.
\end{prop} 
 
Here, we provide a sharper version of this result, which was discussed  in~\cite{sootla2016existence}. We present the proof for completeness.

\begin{theorem} \label{thm:h-mat-lyap}
Let $-A$ be an $\mHp$ matrix with a nonsingular $\cM(A)$. Then the following conditions hold 

(i) there exist positive vectors $v = \begin{pmatrix} v_1 & \dots & v_n \end{pmatrix}^\ast$, $w = \begin{pmatrix} w_1 & \dots & w_n \end{pmatrix}^\ast$ such that $\cM(A) v$, $w^\ast \cM(A)$ are also positive. 

(ii) there exists a diagonal $X$ such that $- (A X + X A^\ast)$ is an $\mHp$ matrix. Moreover, we can choose it as $X = P_w P_v^{-1}$, where $P_v =\diag{v_1, \dots, v_n}$, $P_w =\diag{w_1, \dots, w_n}$.
\end{theorem}
\begin{proof} 
i) By definition $-\cM(A)$ is a Metzler matrix with all eigenvalues $\lambda_i(\cM(A)) \le 0$, since $\cM(A)$ is nonsingular by the premise, $-\cM(A)$ is a Hurwitz Metzler matrix. Hence the claim follows by applying the results from~\cite{rantzer2012distributed}.

ii) According to Proposition~\ref{prop:h-diag-stab} there exists a matrix $X = \diag{ x_1,~\cdots~,x_n}\succ 0$ such that $- AX - X A^\ast \succ 0$. Note that $a_{i i} < 0$ for all $i$, let
\begin{gather*}
(-A X - X A^\ast)_{i j} = - a_{i j} x_j - a_{j i} x_i \\
(\cM(A) X + X \cM(A^\ast))_{i j} = 
\begin{cases} - a_{i j} x_j - a_{j i} x_i  & i = j \\
- |a_{i j}| x_j - |a_{j i}| x_i            & i \ne j
\end{cases}
\end{gather*}

It is straightforward to show that $\cM(A) X + X \cM(A^\ast) \le \cM(- A X - X A^\ast)$, moreover the elements on the diagonal are equal. This means that we can write $\cM(A) X + X \cM(A^\ast) = s I - R_1$, $\cM(- A X - X A^\ast) = s I -R_2$, where the matrices $R_1$ and $R_2$ satisfy $R_1\ge R_2 \ge 0$. According to Weilandt's theorem  $\rho(R_1) \ge \rho(R_2)$ (cf.~\cite{hershkowitz1985lyapunov}). Therefore the minimal eigenvalue of $\cM(A) X + X \cM(A^\ast)$ is smaller or equal to the minimal eigenvalue of $\cM(- A X - X A^\ast)$. This implies that $\cM(- A X - X A^\ast)$ has eigenvalues with positive real part, hence $- A X - X A^\ast$ is an $\mHp$ matrix.

The proof of the second part of the statement is inspired by the proof of Proposition~1 in~\cite{rantzer2012distributed}. Let $X = P_v P_w^{-1}$, then 
\begin{multline*}
(\cM(A) X + X \cM(A^\ast)) w = (\cM(A) v + X  \cM(A^\ast) w ) \gg 0,
\end{multline*}
where the inequality follows since $\cM(A) v$ and $\cM(A^\ast) w$ are positive and $X$ is nonnegative. Hence $S = -\cM(A) X - X \cM(A^\ast)$ is a Metzler matrix and there exists a positive vector $w$ such that $S w$ is negative. This implies that $S$ is a symmetric Hurwitz and Metzler matrix (cf.~\cite{rantzer2012distributed}), which means that $-S$ is positive definite. According to the derivations above $\cM(- A X - X A^\ast) \ge -S$, hence $- A X - X A^\ast$ is an $\mHp$ matrix and consequently it is positive definite.
\end{proof}

If $A$ is additionally an irreducible $\mH$ matrix, then we can use eigenvectors of $\cM(A)$ corresponding to the eigenvalue with the smallest real part as the vectors $v$, $w$. Hence a diagonal Lyapunov function can also be computed using linear algebra as opposed to solving the LMIs~\eqref{eq:lyap-ineq}. 

If an $\mH$ matrix $A$ is ill-conditioned, even although we can guarantee existence of diagonal solutions $P$, $Q$ to~\eqref{lyap-ineq-smor}, in practice, standard solvers frequently fail to provide even a feasible point. This, for example, happens in the yeast glycolysis example in Section \ref{ex:gly}. However, since $A$ is an $\mH$ matrix, we can compute diagonal Lyapunov matrices using linear algebra as shown in Theorem~\ref{thm:h-mat-lyap}. Let $X$ be such that $A X + X A^\ast$ is negative definite. Then we can choose $P_{\rm base}$ satisfying Lyapunov inequality as $P_{\rm base} = X \overline \sigma(B B^\ast)/\underline\sigma(A X + X A^\ast)$, where $\overline\sigma(\cdot)$, and $\underline\sigma(\cdot)$ stand for the maximum and minimum singular values of a matrix, respectively. Similarly, we can compute $Q_{\rm base}$. Now in order to compute a generalised controllability Gramians we solve instead: 
\begin{gather} \label{eq:lyap-ineq-cond}
\begin{aligned}
\min\limits_{P\in \cS}~~~&\tr(P) \\
\text{s. t.: } &A (P+P_{\rm base}) + (P+P_{\rm base}) A^\ast + B B^\ast \preceq 0 \\
 &P+P_{\rm base} \succ 0  
\end{aligned}
\end{gather}
where the initial point is simply $P=0$, which improves numerical properties of the programme. A similar procedure can be derived for $Q$.  

\section{Projection Based Model Reduction of a Linear Noise Approximation\label{s:proj-lna}}
\subsection{A Reduced-Order Model\label{ss:rom}}
In this section we describe how structured \emph{projection} based methods can be applied to model reduction of LNA. By order of the LNA we mean the dimension of the vector $x+\eta$.  We assume we want to preserve the physical interpretation of first $k$ states (chemical species) out of $n$ and that  the full order model takes the following form:
\begin{gather}
\begin{aligned}
  & \dot x = g(x), \\ 
  & \dot \eta = A(x)\eta + B(x) \dot w,  \\
  & y = C (x +\eta) \\
  & x(0) = x_0, \quad \eta(0) = 0, & 
\end{aligned} \label{eq:full-order}
\end{gather}
where $y$ is an artificially introduced ``output'' of our Gaussian process with $C = \begin{pmatrix}I_k & 0_{k,n-k}\end{pmatrix}$. Note that other choices of $C$ are perfectly valid, in particular replacing the identity matrix with an arbitrary full matrix may be desirable. In this setting standard model reduction techniques aim to synthesise a model that approximates~\eqref{eq:full-order} in the input-output sense i.e. the map from $\dot w $ to $y$ but with fewer states. In this work we show how the physical structure of the first $k$ states can be preserved whilst the remaining $n-k$ states form a subsystem which is then reduced to order $r$ where $r<n-k$ but with no physical interpretation.

First we linearise the process~\eqref{eq:full-order} $\eta$ around a steady-state $x_{ss}$ of the mean dynamics and obtain the following SDE:
\begin{gather}\label{model:sde-lti}
\begin{aligned}
\dot \xi &= A \xi + B \dot w \\
y &= C \xi,
\end{aligned}
\end{gather}
where $A = A(x_{ss})$, $B =B(x_{ss})$. Let $\cG$ denote the realisation of the stochastic system \eqref{model:sde-lti}.
Our goal is to choose the transformation $T: (A,B,C)\rightarrow (TAT^{-1},TB, CT^{-1})$ such that by applying averaging to the transformed model, we obtain a reduced order SDE $\cG^{k+r}$, which is stable, while $\|\cG - \hat \cG^{k+r}\|$ is minimised in some norm. As a suitable criterion, we consider the standard $\Htwo$ and $\Hinf$ norms. The $\Htwo$ norm has the interpretation of the integral of the trace of the covariance matrix of the process $y(t)$. The $\Hinf$ norm is maximum over frequencies of the largest eigenvalue values of the spectral density of the process $y(t)$~\cite{aastrom2012introduction}. In order to preserve the structure of the first $k$ state equations we will introduce a structured transformation $T= \diag{I_k, T_2}$.

Once the transformation matrix $T$ has been constructed, we obtain the new states $z = T x$, $\xi = T \eta$,  vector field $\tilde g(z) = T g(T^{-1} x)$, and matrices $\tilde A(z) = T A(T^{-1} z)$, $\tilde B(z) = T B(T^{-1} z)$, and consider the following full order model:
\begin{subequations}\label{transformed-model}
\begin{eqnarray}
\dot z &=& \tilde g(z) \label{transformed-model-z}\\
\dot \xi &=& \tilde A(z) \xi + \tilde B(z) \dot w,\label{transformed-model-xi}\\
y &=& C (z+ \xi)
\end{eqnarray}   
\end{subequations}
Given the system~\eqref{transformed-model}, which is equivalent to the system~\eqref{eq:full-order}, a reduced order model can then be found by applying the averaging method for the LNA~\cite{thomas2012rigorous}. In our examples, we use this averaging method as a theoretical justification for our algorithm, however, in practice we do not compute the Lipschitz constants or verify the separation of time-scales. In this sense our approach should be seen as a heuristic. 

\subsection{Computation of Structured Transformations}\label{ss:comp_struct}
Consider an SDE~\eqref{model:sde-lti} and fix the order of reduced order model to be equal to $k+r$. We will use the structured balanced singular perturbation approach, which results in the following reduced order realisation:
\begin{gather}\label{real:lna-lin-red}
\hat \cG^{k+r} =  \left[\begin{array}{c|c} V (I_{k+r} - A \tilde A) A W &  V (I_{k+r} - A \tilde A) B   \\ \hline\\[-9pt] C(I_{k+r} -  \tilde A A) W   & 0 \end{array} \right], 
\end{gather}
where $\tilde A = W_r (V_r A W_r)^{-1} V_r$ and
\begin{gather}\label{eq:projections}
\begin{aligned}
V &= \diag{I_{k+r}, 0_{k+r,n-k-r}} T, \\
V_r &= \diag{0_{n-k-r,k+r}, I_{n-k-r}} T, \\
W^\ast &= \diag{I_{k+r}, 0_{k+r,n-k-r}} T^{-\ast}, \\
W_r^\ast &= \diag{0_{n-k-r,k+r}, I_{n-k-r}} T^{-\ast}.
\end{aligned}
\end{gather}
All is left is to compute a state-space transformation $T=\diag{I_k, T_2}$, if it exists. 

\textbf{\emph{$\Hinf$ balancing:}} Assume that the drift matrix $A$ is diagonally stable. In order to compute the projections consider the following semidefinite programmes 
\begin{subequations}\label{eq:lyap-struct}
\begin{eqnarray}
\begin{aligned}
\min\limits_{\Sigma_P, P_2}~~~&\tr(P) \\
\text{s. t.: } &A P + P A^\ast + B B^\ast \prec 0 \\
 &P  =\begin{pmatrix}
\Sigma_P & 0_{k,n-k} \\ 
 0_{n-k,k}  & P_2 
\end{pmatrix}\succ 0
\end{aligned}\label{eq:lyap-struct-p}\\
\begin{aligned}
\min\limits_{\Sigma_Q, Q_2}~~~&\tr(Q) \\
\text{s. t.: } &A^\ast Q + Q A + C^\ast C \prec 0 \\
 &Q =\begin{pmatrix}
\Sigma_Q & 0_{k,n-k} \\ 
 0_{n-k,k}  & Q_2 
\end{pmatrix}\succ 0,
\end{aligned}\label{eq:lyap-struct-q}
\end{eqnarray}   
\end{subequations}
where $\Sigma_P$, $\Sigma_Q$ are diagonal matrices. Now we are in the position to state the following result.
\begin{theorem}\label{thm:lmi} Let $A$ be diagonally stable and the matrices $P$ and $Q$ satisfy LMIs~\eqref{eq:lyap-struct-p}--\eqref{eq:lyap-struct-q}. Next, let $T$ be such that $T P T^\ast = T^{-1} Q (T^\ast)^{-1} = \Sigma$, let $\sigma_{i}$ denote the eigenvalues of $(P_2 Q_2)^{1/2}$ such that $\sigma_i \ge \sigma_{i+1}$, and let $\sigma_{r} > \sigma_{r+1}$ for some integer $r< n-k$. Consider the projections defined in~\eqref{eq:projections} and reduced order model $\hat \cG^{k+r}$ defined in~\eqref{real:lna-lin-red}. Then the system $\hat \cG^{k+r}$ is diagonally stable and 
\begin{gather*}
\|\cG - \hat \cG^{k+r}\|_{\Hinf} \le 2\sum\limits_{i= r+1}^{n-k} \sigma_{i}.
\end{gather*} \label{thm:main-hinf}
\end{theorem}

The proof of this theorem is a straightforward application of the results in Subsection~\ref{ss:smor} and Lemma \ref{lem:diag-stab} in Appendix~\ref{s:diag-stab-preserv}, hence we will not formally prove this result. If we are not interested in preserving diagonal stability, then we can relax the structure of the matrices $\Sigma_P$ and $\Sigma_Q$ to be full positive definite matrices. Additionally, if the system is not diagonally stable, but block-diagonal Gramians $P$, $Q$ exist (that is $\Sigma_P$ and $\Sigma_Q$ are full positive definite matrices), then Theorem~\ref{thm:main-hinf} still holds with $\hat \cG^{k+r}$ being a stable realisation.

\textbf{\emph{$\Htwo$ balancing:}} An arguably better way of measuring the norm of a stochastic process is the $\Htwo$ norm, defined as
\begin{equation*}
\|\cG\|_\Htwo^2 := 1/(2\pi) \int_{-\infty}^\infty \tr(\cG^\ast(j\omega) \cG(j\omega)) d\omega.
\end{equation*}
There are many methods for model reduction in the $\Htwo$ norm (cf.~\cite{gugercin:609}), however, none of them can easily be extended to the structured projection techniques. However, there exists a simple heuristic that balances just the generalised controllability Gramian, which is computed by the programme~\eqref{eq:lyap-struct-p}. We were not able to obtain any meaningful error bounds for this heuristic, but the computational results are satisfactory and are demonstrated in what follows. Again the proof is a straightforward application of the results in Subsection~\ref{ss:smor} and Appendix~\ref{s:diag-stab-preserv}

\begin{theorem} Let $A$ be diagonally stable. Let $T$ be such that $T P T^\ast = \Sigma$, let $\sigma_{i}$ denote the eigenvalues of $P_2$ such that $\sigma_i \ge \sigma_{i+1}$, and let $\sigma_{r} > \sigma_{r+1}$ for some integer $r< n-k$. Consider the projections defined in~\eqref{eq:projections} and reduced order model $\hat \cG^{k+r}$ defined in~\eqref{real:lna-lin-red}. Then the realisation $\hat \cG^{k+r}$ is diagonally stable. \label{thm:main-h2}
\end{theorem}

\subsection{On Existence of Structured Transformations}\label{ss:exist-bio}
According to the discussion in the previous section, the structured transformations always exist if the drift matrix $A$, which was found by linearising ~\eqref{transformed-model} about a stable steady state, is diagonally stable. Our numerical computations indicate that for many biochemical networks this condition holds, and in many examples the drift matrix $A$ is actually an $\mH$ matrix. We can only provide some possible reasons for this phenomenon. 

Firstly, some biochemical networks are monotone, which means that the Jacobian of the drift term is a Metzler matrix. Hence it is also an $\mH$ matrix around a stable equilibrium. 

Secondly, as noted in~\cite{sontag2007monotone}, many biochemical networks exhibit a nearly monotone behaviour, meaning that if some interactions are removed then the system becomes monotone. We make a conjecture that near monotonicity is related to having an $\mH$ drift matrix $A(x_{ss})$ in the linearised dynamics. We provide an example of a nearly monotone system, whose linearised dynamics have the $\mH$ drift matrix. This discussion, however, lies outside of the scope of this paper. 

Finally, $\mH$ matrices are closely related to scaled diagonally dominant matrices, which appear in graph theory. Since biochemical networks are often locally stable systems on sparse graphs, it should not be too surprising to find systems with $\mH$ drift matrices in the linearised dynamics.

\section{Examples\label{s:ex}}
\subsection{Comparison of the models\label{ss:err}}
We compare separately the error in the macroscopic dynamics (mean) and the fluctuations (variance), since their dynamic models are decoupled. The error $\E(y-y_r)$ in macroscopic dynamics is computed by perturbing the initial state $x_0$ from the steady-state $x_{s s}$ and measured in $L_1$, $L_2$ and $L_{\infty}$ norms. 

A comparison in terms of the fluctuations $\eta$ is performed by computing the covariance matrix of the outputs $y$ and $y_r$. For the full order model this matrix is computed as 
\[
\cov(y)= C \cov(\eta \eta^\ast) C^\ast =  C P C^\ast,
\]
where $P$ satisfies the Lyapunov equation~\eqref{eq:cov-lna}. Similarly, the covariance matrix for the reduced order models $\cov(y_r)$ can be computed. Note that the $L_2$ error of the outputs serves as a lower bound on the $\Htwo$ norm computation.

\subsection{Toy Example. \label{ex:cov}}
The first network we consider consists of only four species, see Figure~\ref{fig:toy-ex}.  One can interpret the species $S_1$ and $S_3$ as mRNA, and $S_2$ and $S_4$ as the corresponding proteins. 
   \begin{align*}
      &\dot m_i = \frac{c_{i 1}}{1 + p_j^2}-c_{i 2} m_i, \quad i\in \left\{1,2\right\}, \\
      &\dot p_i = c_{i 3} m_i-c_{i 4} p_i.
   \end{align*}
 where $c_{i1}$ are constants, $m_1$, $m_2$ are species $S_1$, $S_3$, $p_1$, $p_2$ are species $S_2$, $S_4$. We compare the simulation results for the full order model, the reduced order model obtained by~\cite{thomas2012rigorous}, and the reduced order model obtained from reduction according to the configurations in Figure~\ref{fig:toy-ex} with parameters
\[
  c_{1\cdot}= c_{2\cdot}  = \begin{pmatrix} 3 & 4 & 1 & 0.2  \end{pmatrix}^\ast.
\]
 \begin{figure}[t]
    \centering
   \subfigure[A configuration for reducing one state]{\includegraphics[width=0.33\columnwidth]{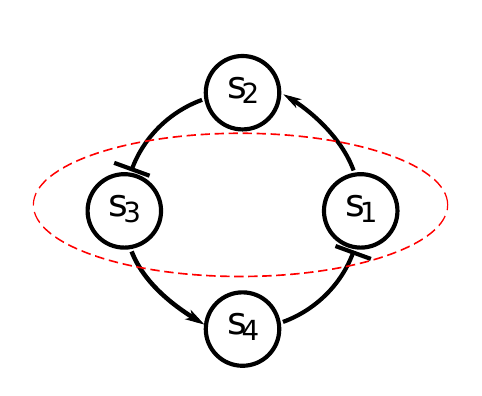} \label{fig:toy-ex-1}}\qquad
   \subfigure[A configuration for reducing one state]{\includegraphics[width=0.33\columnwidth]{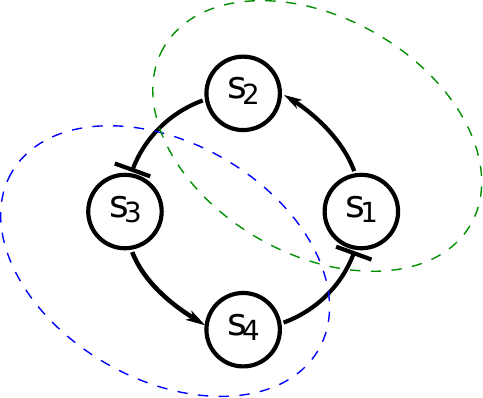} \label{fig:toy-ex-2}}
\caption{Toy Example. In the configuration in the left panel, species $S_1$ and $S_3$ are grouped together, while reducing one state. In the configuration in the left panel, species $S_1$, $S_2$, and $S_3$, $S_4$ are grouped together, while reducing one state from each group.}\label{fig:toy-ex}
\end{figure}
We initiate simulations from $x_0 = \begin{pmatrix}
1 & 10 & 1 & 1
\end{pmatrix}^\ast$, which lies in the domain of attraction of the steady-state $x_{ss} = \begin{pmatrix}
0.2889 & 3.4611 & 0.0578 & 0.6922
\end{pmatrix}^\ast$. We compute the initial covariance $P_0$ from the linearisation around $x_0$ and computing the covariance at this point. We compute the projections with respect to the linearised model at steady-state $x_{ss}$. This linearised model has a Metzler drift matrix, hence the diagonal Gramian always exists and we can test different configurations for model reduction.

\begin{figure}[t]
\centering
\subfigure[Errors in the mean number of species $S_1$ (dashed) and $S_3$ (solid) between full and reduced models. ]{\includegraphics[width=0.45\columnwidth]{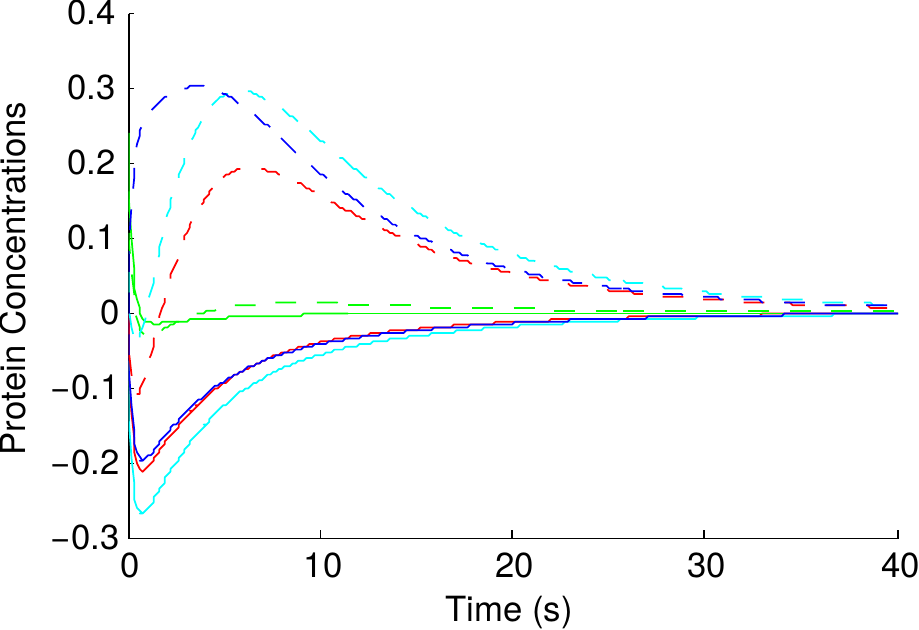} \label{fig:means}}
 \subfigure[Errors in the variance of fluctuations in the number of species $S_1$  between full and reduced models. ]{\includegraphics[width=0.45\columnwidth]{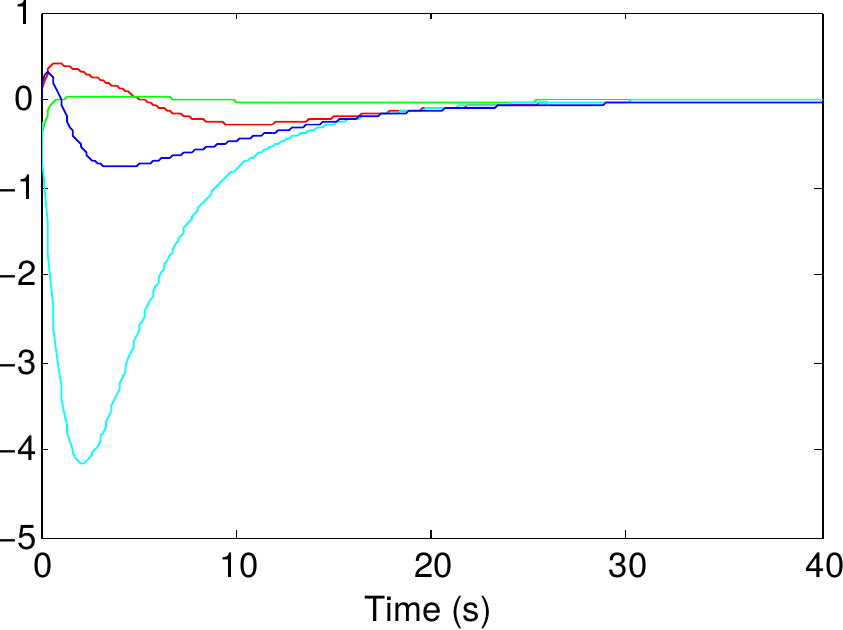} \label{fig:cov11}}
 \subfigure[Errors in the covariance in fluctuations in the number of species $S_1$ and $S_3$  between full and reduced models.]{\includegraphics[width=0.45\columnwidth]{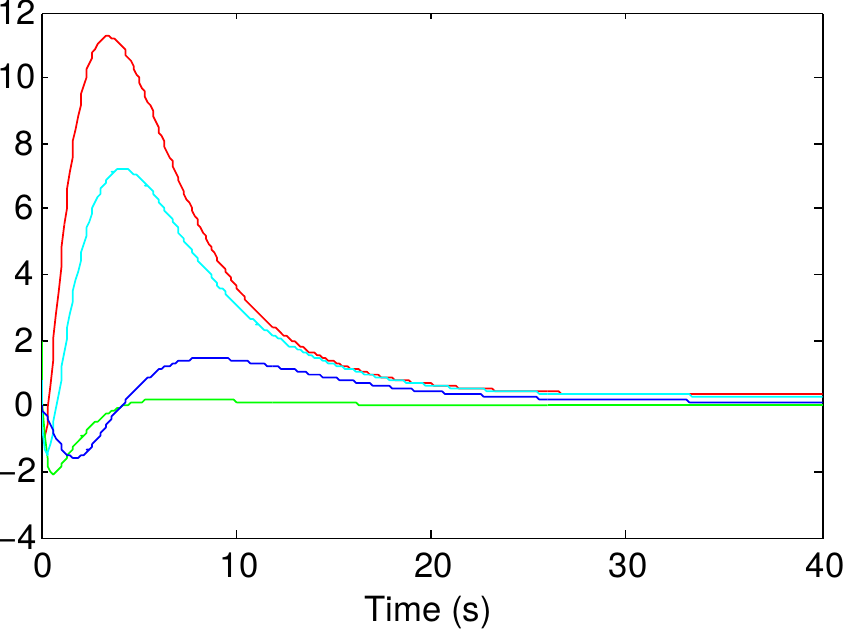} \label{fig:cov21}}
 \subfigure[Errors in the variance of fluctuations in the number of species $S_3$  between full and reduced models.]{\includegraphics[width=0.45\columnwidth]{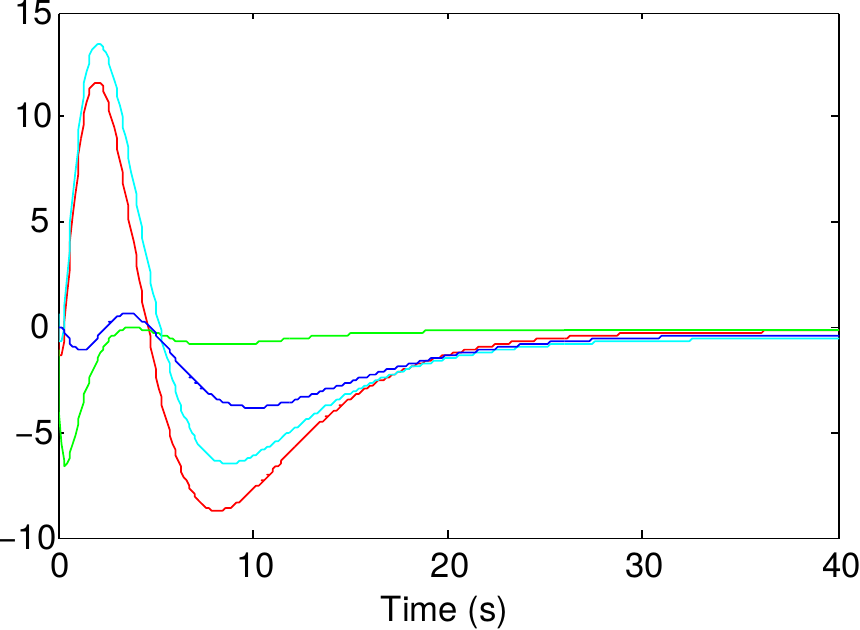} \label{fig:cov22}}
\caption{In all figures, the blue lines correspond to the evolution of the error between full and reduced model obtained by $\Hinf$ balancing according to the configuration in Figure~\ref{fig:toy-ex-1}, the red lines correspond to the evolution of error between full and reduced order model obtained using~\cite{thomas2012rigorous}, the green lines correspond between full and reduced model obtained by $\Hinf$ balancing according to the configuration in Figure~\ref{fig:toy-ex-2}. The same correspondence is valid for the cyan lines and $\Htwo$ balancing.} 
\label{fig:stoch}
\end{figure}

Comparisons between the various reduced and full models are depicted in~Figure~\ref{fig:stoch}. It can be seen that the method from~\cite{thomas2012rigorous} always performs worse than the $\Hinf$ balancing method. This example also highlights the importance of selecting which parts of the system to reduce: in the configuration depicted in Figure~\ref{fig:toy-ex-1} we reduce one state, but it is worse than reducing two states according to configuration depicted in Figure~\ref{fig:toy-ex-2}. This happens since we do not respect the topology. The $\Htwo$ balancing does poorly compared to~\cite{thomas2012rigorous} even though reduction for the linearised models gave a similar result to $\Hinf$ balancing. 
%

\subsection{Kinetic Model of Yeast Glycolysis. \label{ex:gly}}
This 12-state model was published in~\cite{van2012testing} and it consists of twelve metabolites and four boundary fluxes. We set the state-space of the model to:
\begin{gather*}
x = \left(\begin{array}{cccccc}\text{GLCi} &  \text{G6P} &  \text{F6P} &  \text{F16P} &  \text{TRIO} &  \text{BPG} \end{array} ...\right.   \\
\left.\begin{array}{cccccc}    \text{P3G} & \text{P2G} & \text{PEP} &  \text{PYR} &  \text{ACALD} &  \text{NADH} \end{array}\right)^\ast 
\end{gather*}

We model the network's response to the change of glucose in the system as in~\cite{rao2014model}.  
We treat levels of $\mathrm{ATP}$ and glycose $\mathrm{GLCo}$ as control inputs. At time zero we change the levels of $\mathrm{ATP}$ and $\mathrm{GLCo}$ from $3$ to $1.5$ and $0.25$ to $5$ respectively. Let $x_0$ be the steady state with $\mathrm{ATP} =3$ and $\mathrm{GLCo}=0.25$, while $x_{s s}$ be the steady state $\mathrm{ATP} =1.5$ and $\mathrm{GLCo}=5$.
\begin{figure}[t]
  \centering
  \includegraphics[width=0.8\columnwidth]{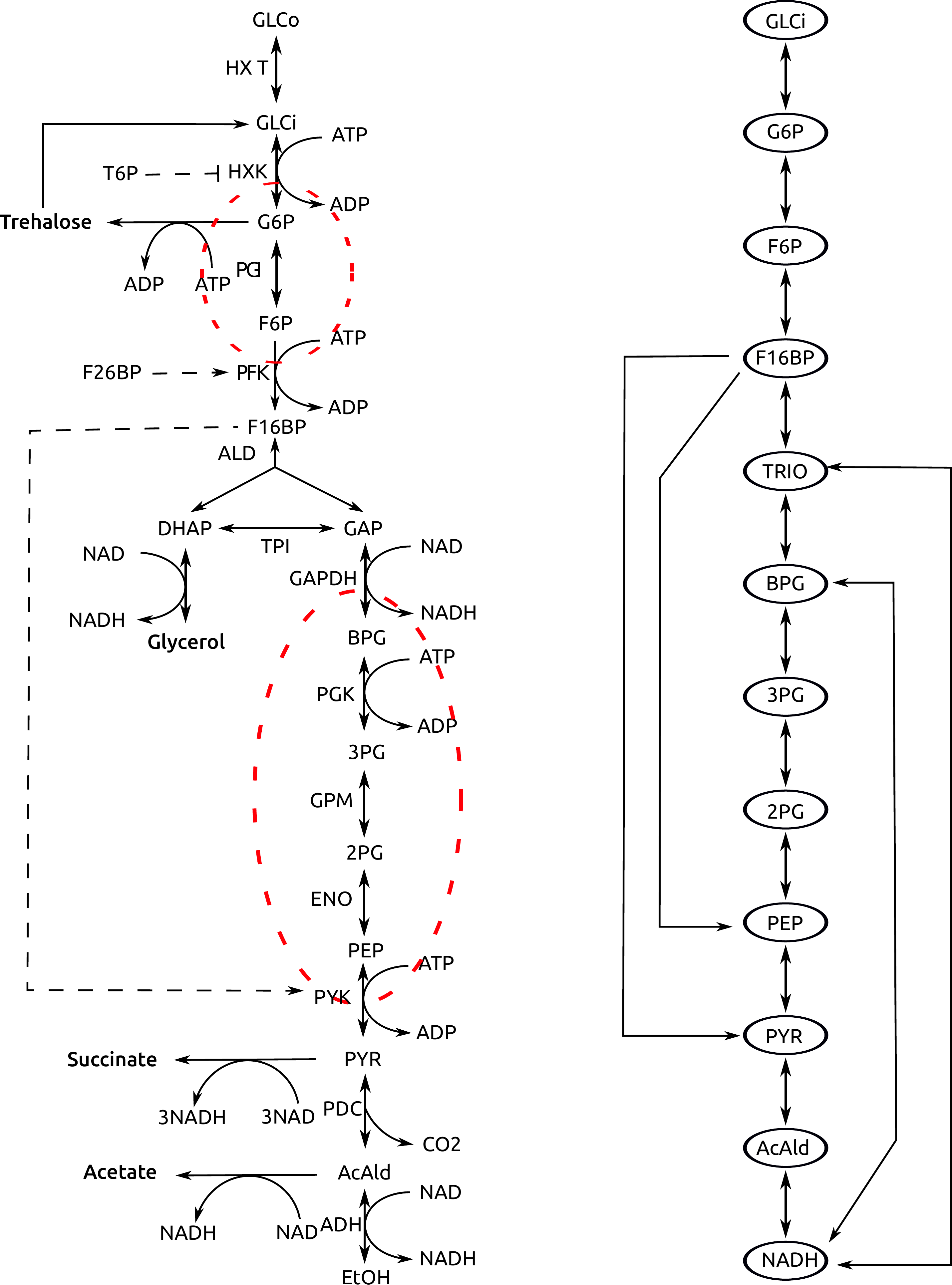} 
  \caption{Kinetic model of yeast glycolysis. In the left panel the biochemical graph is depicted. In the right panel a graph of dynamic interactions between metabolites. The red ellipses mark the reduction regions of interest.}\label{fig:glycerol}
\end{figure}
%
%

We refer the reader to~\cite{van2012testing} for a complete description of the model, however, we mention that the drift matrices of the linearised dynamics around $x_0$ and $x_{ss}$ are stable $\mH$ matrices. Hence we can compute diagonal Lyapunov functions by solving semidefinite programmes. We pick two groups of species to reduce \{BPG, P3G, P2G, PEP\} and \{GLCi, G6P, F6P\}, see Figure \ref{fig:glycerol}, however, we consider them separately. Meaning that the Gramians $P$ and $Q$ have three blocks not two. As demonstrated on the toy example violating topological constraints in the graph (mixing these two groups) can result in reduced order models of worse quality.

The error is computed by simulating the resulting reduced order models and comparing them as described at the beginning of the section. The results are presented in Table~\ref{tab:gly-red} for various reduction configurations. We apply~\cite{thomas2012rigorous} to metabolite concentrations, while using the proposed method we try to lump those metabolites in one state, so that the number of reduced states is similar in both cases. The first two rows of each sub-table in Table~\ref{tab:gly-red} can be compared directly, and it is clear that the proposed method performs better in terms of quality than~\cite{thomas2012rigorous}. 
The proposed methods are also more flexible in terms of reduction choices. In the forth row of Table~\ref{tab:gly-red}-B, the region \{BPG-PEP\} contains four metabolites; however, we reduced only two states after computing the state-space transformation. In the fifth row, in the region  \{GLCi-F6P\}, which contains three metabolites, we reduce just one state and this provides us with the best model among all the reduction attempts. Finally, the results in Table~\ref{tab:gly-red}-C indicate that the $\Htwo$ balancing outperforms the $\Hinf$ balancing on this example.
\begin{table}[t]
\centering
\caption{Reduction of the glycolysis model. The error $\E\|y(t) -y_r(t)\|$ in different norms, where $y$ and $y_r$ are the trajectory of the full and reduced order models, respectively} \label{tab:gly-red} 
  \begin{tabular}{cccc}
  \multicolumn{4}{c}{ \sc Table~\ref{tab:gly-red}-A. Approximation results obtained by using~\cite{thomas2012rigorous}} \\
  Reduced States $\backslash$ Error & $L_1$ & $L_2$ & $L_{\infty}$\\
\hline
\hline
  F6P, 2PG, PEP              & $1.2143$ & $0.7490$ & $0.9782$ \\
  F6P, 3PG, 2PG, PEP         & $1.5740$ & $1.0674$ & $1.5582$ \\
\hline\\[-3pt]    
  \end{tabular}
  \begin{tabular}{ccccc}
  \multicolumn{5}{c}{ \sc Table~\ref{tab:gly-red}-B. Reduction by $\{k_1, k_2\}$ states } \\
  \multicolumn{5}{c}{ \sc in every region using the $\Hinf$ balancing} \\
    Lumped Region(s)         & $\{k_1, k_2\}$  & $L_1$  & $L_2$  & $L_{\infty}$ \\
  \hline
  \hline
  \{G6P, F6P\}, \{3PG--PEP\} & $\{1, 2\}$ & $1.1816$ & $0.7864$ & $1.0118$ \\
  \{G6P, F6P\}, \{BPG--PEP\} & $\{1, 3\}$ & $1.4176$ & $0.7273$ & $0.8702$ \\
  \{G6P, F6P\}, \{3PG--PEP\} & $\{1, 1\}$ & $0.3818$ & $0.2527$ & $0.3243$ \\
  \{G6P, F6P\}, \{BPG--PEP\} & $\{1, 2\}$ & $1.4129$ & $0.7242$ & $0.8651$ \\
\hline\\[-3pt]
  \end{tabular}
  \begin{tabular}{ccccc}
  \multicolumn{5}{c}{ \sc Table~\ref{tab:gly-red}-C. Reduction by $\{k_1, k_2\}$ states } \\
  \multicolumn{5}{c}{ \sc in every region using the $\Htwo$ balancing} \\
    Lumped Region(s)         & $\{k_1, k_2\}$  & $L_1$  & $L_2$  & $L_{\infty}$ \\
  \hline
  \hline
  \{G6P, F6P\}, \{3PG--PEP\} & $\{1, 2\}$ & $1.0906$ & $0.6857$ & $0.8815$ \\
  \{G6P, F6P\}, \{BPG--PEP\} & $\{1, 3\}$ & $1.1002$ & $0.6936$ & $0.8954$ \\
  \{G6P, F6P\}, \{3PG--PEP\} & $\{1, 1\}$ & $0.2935$ & $0.1587 $ & $0.1979$ \\
  \{G6P, F6P\}, \{BPG--PEP\} & $\{1, 2\}$ & $0.2104$ & $0.1080$ & $0.1313$ \\
\hline
  \end{tabular}  
\end{table}

\section{Conclusion}
In this paper, we studied model order reduction of the Linear Noise Approximation of the Chemical Master Equation. We showed that a recently proposed time-scale separation method results in a reduced order model, which converges in the mean-square sense to the slow dynamics of the LNA. We then considered the application of structure preserving, projection-based model reduction to the  LNA. One of the bottlenecks of projection-based methods, is existence of the projectors, which cannot be always guaranteed. We were able to provide sufficient conditions that describe when such projectors exist. Furthermore, these are spectral conditions on the drift matrix of the linearised dynamics, hence they are easy to check. 

As a straightforward extension of this approach, we may consider model reduction for time varying SDEs using for example~\cite{san+04}. This may provide better quantitative approximations of LNAs. However, there are deeper issues with the LNA itself. If the underlying CME is bimodal (in some cases this implies, for example, that the deterministic model of macroscopic reaction rates is bistable), then LNA and hence our approximation procedure will not capture this phenomenon. Therefore one needs to derive projection-based reduction methods for the CME or at least for the Chemical Langevin Equation, which is a nonlinear SDE. In this case, it is perhaps possible to use nonlinear balancing tools~\cite{Scherpen93}, which are based on energy functions. The controllability energy function is identical to the action functional used the fastest escape problems in the small noise limit~\cite{freidlin2012random}. It remains to establish, however, if this action functional can be used for model reduction.

\appendices
\section{Technical Lemmas for the Proof  of Theorem~\ref{thm:conv} \label{app:convergence}}
\begin{lem}[Gronwall Lemma]
\label{prop:gronwall}
Let $g(t)$, $t\in[0,T]$ be a non-negative, continuous real-valued function that satisfies
\begin{equation*}
g(t) \le C + K \int_0^tg(s)ds
\end{equation*}
for all $t\in [0,a]$ where $C$ and $K$ are positive constants. Then it follows that for all $t\in [0,a]$,
\begin{equation*}
g(t)\le Ce^{Kt}.
\end{equation*} 
\end{lem}
\begin{lem}\label{lem:lip-bounds-proof}
Let $Z\subset \R$ be a compact set, then any polynomial function $p:Z\rightarrow  \R$ is Lipschitz.
\end{lem}
\begin{proof}
As $Z$ is compact by definition we can assume that $p:[a,b]\rightarrow \R$, further, $p$ is a polynomial and therefore in $\cC^{\infty}$, thus its first derivative $p'$ exists at every point and is continuous, Moreover $p':[a,b]\rightarrow \R$ is bounded, i.e. there exists a positive constant $K$ such that $\|p'(x)\| \le K$ $\forall x \in [a,b]$. Given any $x,y \in R$ such that $a\le y \le x \le b$, by application of the Mean Value Theorem it follows that 
\begin{equation*}
\|p(x)-p(y)\| = \|p'(c)\| \|x-y\| \le K\|x-y\|
\end{equation*}
and so $K$ is the Lipschitz constant of $p$ on $[a,b]$. The extension to the multivariable case is straightforward as the Multivariable Mean Value Theorem can be applied and all partial derivatives of a $C^{\infty}$ function are bounded on a compact set.
\end{proof}
\begin{lem} \label{lem:c1-bounds} 
Let 
\begin{align*}
C_1(t) = \int_{0}^{t} A_{1 2} A_{2 2}^{-1}(\varepsilon^{-1/2} A_{2 2} \eta_2   + A_{2 1} \eta_1 + B_2 \dot w) d \tau
\end{align*}
and assume that $A_{22}(x)$ is a stable matrix for all $x \in \cD$ where $\cD$ is a given connected domain. Then the following bound holds:  $\E\|C_1(t)\|^2  = O(\varepsilon)$.
\end{lem}
\begin{proof}
We have the following chain of inequalities
\begin{multline*}
C_1(t) = \int_{0}^{t} A_{1 2} A_{2 2}^{-1}(\varepsilon^{-1/2} A_{2 2} \eta_2   + A_{2 1} \eta_1 + B_2 \dot w) d \tau = \\
\varepsilon^{1/2} \int_{0}^{t} A_{1 2} A_{2 2}^{-1}(\varepsilon^{-1/2}  A_{2 1} \eta_1 +\varepsilon^{-1} A_{2 2} \eta_2   + \varepsilon^{-1/2} B_2 \dot w) d \tau = \\
\varepsilon^{1/2} \int_{0}^{t} A_{1 2} A_{2 2}^{-1} \dot \eta_2 dt = \varepsilon^{1/2} \int_{0}^{t} A_{1 2} A_{2 2}^{-1} d \eta_2 
\end{multline*}

Now let us bound $\E\|C_1(t)\|^2$ as follows:
\begin{gather*}
\E\|C_1(t)\|^2 = \varepsilon \E \left\|\int_{0}^{t} A_{1 2} A_{2 2}^{-1} d \eta_2 \right\|^2 \le 
\varepsilon C_2 \cov(\eta_2(t)),
\end{gather*}
where $C_2$ is such that $C_2 \ge \|A_{1 2} A_{2 2}^{-1}\|$ for all $0 \le \tau \le t$. It is left to verify that the covariance of $\eta_2(t)$ is bounded with $\varepsilon\rightarrow 0$. This is easily verified by considering the (2,2) block of $\dot{P}$ from Proposition \ref{prop:cov} and noting that as $\varepsilon\rightarrow 0$ the element becomes a Lyapunov equation for a stable $A_{22}$ matrix. 
%
%
%
\end{proof}

\section{Preservation of Diagonal Stability \label{s:diag-stab-preserv} }
In order to prove the Theorems~\ref{thm:lmi} and~\ref{thm:main-h2} we need to derive the following lemma for the preservation of diagonal stability.

\begin{lem}\label{lem:diag-stab}
Let $F$ be a diagonally stable matrix and consider the solution to the following Lyapunov inequality 
\begin{multline*}
\begin{pmatrix}
F_{1 1} & F_{1 2} \\[1pt]
F_{2 1} & F_{2 2}
\end{pmatrix} \begin{pmatrix}
\Sigma_{1} & 0 \\[1pt]
0     & P_{2}
\end{pmatrix} + \\ \begin{pmatrix}
\Sigma_{1} & 0 \\[1pt]
0     & P_{2}
\end{pmatrix} \begin{pmatrix}
F_{1 1}^\ast & F_{2 1}^\ast \\[1pt]
F_{1 2}^\ast & F_{2 2}^\ast
\end{pmatrix} \prec - X,
\end{multline*}
where $\Sigma_1$ is a diagonal, and $P_2$, $X$ are full positive definite matrices. Let $W$ be an invertible matrix partitioned $W^\ast = \begin{pmatrix}
w_1 & w_2
\end{pmatrix}$, where $w_1$ has $r$ columns. Let also $V = W^{-1} = \begin{pmatrix}
v_1 & v_2
\end{pmatrix}$, where $v_1$ has $r$ columns. Furthermore, let $V$ be such that $V P_2 V^\ast =\Sigma_2 = \begin{pmatrix}
\Sigma_{2,1} & 0 \\
0            & \Sigma_{2,2}
\end{pmatrix}$, where $\Sigma_{2,1}$ is an $r$ by $r$ diagonal matrix and $\Sigma_{2,2}$ is a diagonal matrix of an appropriate dimension. Then the matrix 
\begin{multline*}
F_r  = \begin{pmatrix}
F_{1 1} & F_{1 2} w_1 \\
v_1^\ast F_{2 1} & v_1^\ast F_{2 2} w_1
\end{pmatrix} - \\
\begin{pmatrix}
F_{1 2} \\
v_1^\ast F_{2 2} 
\end{pmatrix} w_2  (v_2^\ast F_{2 2} w_2)^{-1} v_2^\ast \begin{pmatrix}
 F_{2 1}^\ast \\  (F_{2 2} w_1)^\ast
\end{pmatrix}^\ast
\end{multline*}
is diagonally stable.
\end{lem}
\begin{proof}
By the premise we have that 
\begin{multline*}
\begin{pmatrix}
F_{1 1}          & F_{1 2} w_1           & F_{1 2} w_2 \\
v_1^\ast F_{2 1} & v_1^\ast F_{2 2} w_1  & v_1^\ast F_{2 2} w_2 \\
v_2^\ast F_{2 1} & v_2^\ast F_{2 2} w_1  & v_2^\ast F_{2 2} w_2 
\end{pmatrix} \begin{pmatrix}
\Sigma_{1} & 0  & 0\\
0     & \Sigma_{2,1} & 0\\
0 & 0 &\Sigma_{2,2}
\end{pmatrix} + \\ \begin{pmatrix}
\Sigma_{1} & 0  & 0\\
0     & \Sigma_{2,1} & 0\\
0 & 0 &\Sigma_{2,2}
\end{pmatrix} \begin{pmatrix}
F_{1 1}^\ast & F_{2 1}^\ast v_1              & F_{2 1}^\ast v_2 \\
w_1^\ast F_{1 2}^\ast & w_1^\ast F_{2 2}^\ast v_1  & w_1^\ast F_{2 2}^\ast v_2 \\
w_2^\ast F_{1 2}^\ast & w_2^\ast F_{2 2}^\ast v_1  & w_2^\ast F_{2 2}^\ast v_2 
\end{pmatrix} \prec \\
-\begin{pmatrix}
I & 0 \\
0 & V^\ast
\end{pmatrix} X \begin{pmatrix}
I & 0 \\
0 & V
\end{pmatrix}.
\end{multline*}
Let $H = (V^\ast F W)^{-1}$ and partitioned conformally to $V^\ast F W$, that is 
\begin{gather*}
H = \begin{pmatrix}
H_{1 1} & H_{1 2} & H_{1 3}\\
H_{2 1} & H_{2 2} & H_{2 3} \\
H_{3 1} & H_{3 2} & H_{3 3} 
\end{pmatrix},
\end{gather*}
where it is straightforward to verify that $F_r  = \begin{pmatrix}
H_{1 1} & H_{1 2}       \\
H_{2 1} & H_{2 2}  
\end{pmatrix}$. Then we have that 
\begin{multline*}
\begin{pmatrix}
H_{1 1} & H_{1 2} & H_{1 3}\\
H_{2 1} & H_{2 2} & H_{2 3} \\
H_{3 1} & H_{3 2} & H_{3 3} 
\end{pmatrix}
 \begin{pmatrix}
\Sigma_{1} & 0  & 0\\
0     & \Sigma_{2,1} & 0\\
0 & 0 &\Sigma_{2,2}
\end{pmatrix} + \\ \begin{pmatrix}
\Sigma_{1} & 0  & 0\\
0     & \Sigma_{2,1} & 0\\
0 & 0 &\Sigma_{2,2}
\end{pmatrix} \begin{pmatrix}
H_{1 1}^\ast & H_{2 1}^\ast & H_{3 1}^\ast\\
H_{1 2}^\ast & H_{2 2}^\ast & H_{3 2}^\ast \\
H_{1 3}^\ast & H_{2 3}^\ast & H_{3 3}^\ast 
\end{pmatrix} \prec \\
- H \begin{pmatrix}
I & 0 \\
0 & V^\ast
\end{pmatrix} X \begin{pmatrix}
I & 0 \\
0 & V
\end{pmatrix} H^\ast.
\end{multline*}

It is immediate that the matrix
\begin{multline*}
\begin{pmatrix}
H_{1 1} & H_{1 2}       \\
H_{2 1} & H_{2 2}  
\end{pmatrix} \begin{pmatrix}
\Sigma_{1} & 0 \\
0     & \Sigma_{2,1}
\end{pmatrix} + \begin{pmatrix}
\Sigma_{1} & 0  \\
0     & \Sigma_{2,1}
\end{pmatrix} \begin{pmatrix}
H_{1 1}^\ast & H_{2 1}^\ast \\
H_{1 2}^\ast & H_{2 2}^\ast
\end{pmatrix}
\end{multline*}
is negative definite, which completes the proof.
\end{proof}
\bibliographystyle{plain}
\bibliography{bibl_reduction}

\begin{thebibliography}{10}

\bibitem{AndP12}
J.~Anderson and A.~Papachristodoulou.
\newblock A decomposition technique for nonlinear dynamical system analysis.
\newblock {\em IEEE Trans Autom Control}, 57(6):1516--1521, 2012.

\bibitem{andersondecentralised}
J.~Anderson and A.~Sootla.
\newblock Decentralised {H2} norm estimation and guaranteed error bounds using
  structured gramians.
\newblock In {\em Proc MTNS}, Groningen, Netherlands, July. 2014.

\bibitem{AntoulasBook}
A.~C. Antoulas.
\newblock {\em Approximation of Large-Scale Dynamical Systems (Advances in
  Design and Control)}.
\newblock SIAM, 2005.

\bibitem{arcak2011diagonal}
M.~Arcak.
\newblock Diagonal stability on cactus graphs and application to network
  stability analysis.
\newblock {\em IEEE Trans Autom Control}, 56(12):2766--2777, 2011.

\bibitem{aastrom2012introduction}
K.J. {\AA}str{\"o}m.
\newblock {\em Introduction to stochastic control theory}.
\newblock Courier Corporation, 2012.

\bibitem{cao2008slow}
Y.~Cao and L.~Petzold.
\newblock Slow-scale tau-leaping method.
\newblock {\em Computer methods in applied mechanics and engineering},
  197(43):3472--3479, 2008.

\bibitem{freidlin2012random}
M.~Freidlin and A.D. Wentzell.
\newblock {\em Random perturbations of dynamical systems}, volume 260.
\newblock Springer, 2012.

\bibitem{gillespie1977exact}
D.~T. Gillespie.
\newblock Exact stochastic simulation of coupled chemical reactions.
\newblock {\em J Phys Chem}, 81(25):2340--2361, 1977.

\bibitem{Gil01}
D.T. Gillespie.
\newblock Approximate accelerated stochastic simulation of chemically reacting
  systems.
\newblock {\em J Chem Phys}, 115(4):1716--1733, 2001.

\bibitem{gugercin:609}
S.~Gugercin, A.~C. Antoulas, and C.~Beattie.
\newblock $\mathcal{H}_2$ model reduction for large-scale linear dynamical
  systems.
\newblock {\em SIAM J. Matrix Anal. Appl.}, 30(2):609--638, 2008.

\bibitem{hartmann2011balanced}
C.~Hartmann.
\newblock Balanced model reduction of partially observed langevin equations: an
  averaging principle.
\newblock {\em Math Computer Modelling Dynamical Systems}, 17(5):463--490,
  2011.

\bibitem{herathmodel}
N.~Herath, A.~Hamadeh, and D.~Del~Vecchio.
\newblock Model reduction for a class of singularly perturbed stochastic
  differential equations.
\newblock In {\em Proc Am Control Conf}, Chicago, IL, July. 2015.

\bibitem{hershkowitz1985lyapunov}
D.~Hershkowitz and H.~Schneider.
\newblock Lyapunov diagonal semistability of real {H}-matrices.
\newblock {\em Linear algebra and its applications}, 71:119--149, 1985.

\bibitem{ishizaki2015clustered}
T.~Ishizaki, K.~Kashima, A.~Girard, J.~Imura, L.~Chen, and K.~Aihara.
\newblock Clustered model reduction of positive directed networks.
\newblock {\em Automatica}, 59:238--247, 2015.

\bibitem{kang2013separation}
H.-W. Kang, T.G. Kurtz, et~al.
\newblock Separation of time-scales and model reduction for stochastic reaction
  networks.
\newblock {\em Annals Applied Prob}, 23(2):529--583, 2013.

\bibitem{khas1968aver_c}
R.Z. Khasminskii.
\newblock On the principle of averaging for the it{\^o} stochastic differential
  equations.
\newblock {\em Kybernetika}, 4:260--279, 1968.

\bibitem{kokotovic1999singular}
P~Kokotovic, H.K. Khalil, and J.~O'Reilly.
\newblock {\em Singular perturbations methods in control: analysis and design}.
\newblock {SIAM}, 1999.

\bibitem{KrylovBogolyubov1937}
N.M. Krylov and N.N. Bogolyubov.
\newblock Les proprietes ergodiques des suites des probabilites en chaine.
\newblock {\em C. R. Math. Acad. Sci.}, 204:1454--1546, 1937.

\bibitem{kurtz1992averaging}
T.G. Kurtz.
\newblock Averaging for martingale problems and stochastic approximation.
\newblock In {\em Applied Stochastic Analysis}, pages 186--209. Springer, 1992.

\bibitem{LiuK10}
S.-J. Liu and M.~Krstic.
\newblock Continuous-time stochastic averaging on the infinite interval for
  locally {L}ipschitz systems.
\newblock {\em SIAM J Control Optimization}, 48(5):3589--3622, 2010.

\bibitem{LiuK10a}
S.-J. Liu and M.~Krstic.
\newblock Stochastic averaging in continuous time and its applications to
  extremum seeking.
\newblock {\em IEEE Trans Autom Control}, 55(10):2235--2250, 2010.

\bibitem{liu1989singular}
Y.~Liu and B.D.O. Anderson.
\newblock Singular perturbation approximation of balanced systems.
\newblock {\em Int J Control}, 50(4):1379--1405, 1989.

\bibitem{MonTC14}
N.~Monshizadeh, H.~L. Trentelman, and M.~K. Camlibel.
\newblock Projection-based model reduction of multi-agent systems using graph
  partitions.
\newblock {\em IEEE Transactions on Control of Network Systems}, 1(2):145--154,
  2014.

\bibitem{Moore1981}
B.~Moore.
\newblock Principal component analysis in linear systems: Controllability,
  observability, and model reduction.
\newblock {\em IEEE Trans. Autom. Control}, 26(1):17--32, Feb 1981.

\bibitem{MunK06}
B.~Munsky and M.~Khammash.
\newblock The finite state projection algorithm for the solution of the
  chemical master equation.
\newblock {\em The Journal of chemical physics}, 124(4):044104, 2006.

\bibitem{rantzer2012distributed}
A.~Rantzer.
\newblock Distributed control of positive systems.
\newblock {\em arXiv preprint arXiv:1203.0047}, 2012.

\bibitem{rao2014model}
S.~Rao, A.~van~der Schaft, K.~van Eunen, B.M. Bakker, and B.~Jayawardhana.
\newblock A model reduction method for biochemical reaction networks.
\newblock {\em BMC Syst Biol}, 8(1):52, 2014.

\bibitem{Reiss2003}
M.~Reiss.
\newblock {\em Stochastic differential equations. Lecture Notes, Humboldt
  University Berlin}.
\newblock 2003.

\bibitem{Sandberg09}
H.~Sandberg and R.~M. Murray.
\newblock Model reduction of interconnected linear systems.
\newblock {\em Optimal control applications \& methods}, 30(3):225--245, 2009.

\bibitem{san+04}
H.~Sandberg and A.~Rantzer.
\newblock Balanced truncation of linear time-varying systems.
\newblock {\em IEEE Trans. Autom. Control}, 49(2):217--229, February 2004.

\bibitem{Scherpen93}
J.M.A. Scherpen.
\newblock Balancing for nonlinear systems.
\newblock {\em Systems and Control Letters}, 21:143--153, 1993.

\bibitem{sontag2007monotone}
E~D Sontag.
\newblock Monotone and near-monotone biochemical networks.
\newblock {\em Systems and Synthetic Biology}, 1(2):59--87, 2007.

\bibitem{sootla2014projectionI}
A.~Sootla and J.~Anderson.
\newblock On projection-based model reduction of biochemical networks--part i:
  The deterministic case.
\newblock In {\em Proc Conf Decison Control}, pages 3615--3620, dec 2014.

\bibitem{sootla2014projectionII}
A.~Sootla and J.~Anderson.
\newblock On projection-based model reduction of biochemical networks--part ii:
  The stochstic case.
\newblock In {\em Proc Conf Decison Control}, pages 3621--3626, dec 2014.

\bibitem{sootla2016existence}
A.~Sootla and J.~Anderson.
\newblock On existence of solutions to structured lyapunov inequalities.
\newblock In {\em submission to the Ame Control Conf}, 2016.

\bibitem{Sootla2012positive}
A.~Sootla and A.~Rantzer.
\newblock Scalable positivity preserving model reduction using linear energy
  functions.
\newblock In {\em Proc. Conf. Decision Control}, pages 4285--4290, Dec. 2012.

\bibitem{thomas2012rigorous}
P.~Thomas, R.~Grima, and A.V. Straube.
\newblock Rigorous elimination of fast stochastic variables from the linear
  noise approximation using projection operators.
\newblock {\em Phys Review E}, 86(4):041110, 2012.

\bibitem{ThoSG12}
P.~Thomas, A.V. Straube, and R~Grima.
\newblock The slow-scale linear noise approximation: an accurate, reduced
  stochastic description of biochemical networks under timescale separation
  conditions.
\newblock {\em BMC Syst Biol}, 6(1):39, 2012.

\bibitem{van2012testing}
K~van Eunen, J.~Kiewiet, H.V. Westerhoff, and B.M. Bakker.
\newblock Testing biochemistry revisited: how in vivo metabolism can be
  understood from in vitro enzyme kinetics.
\newblock {\em PLoS Comp. Biol.}, 8(4):e1002483, 2012.

\bibitem{varga1976recurring}
R.S. Varga.
\newblock On recurring theorems on diagonal dominance.
\newblock {\em Linear Algebra and its Applications}, 13(1):1--9, 1976.

\bibitem{willems1976lyapunov}
J.C. Willems.
\newblock Lyapunov functions for diagonally dominant systems.
\newblock {\em Automatica}, 12(5):519--523, 1976.

\end{thebibliography}
\end{document}